\documentclass[12pt,reqno]{amsart}
\usepackage{amssymb,amsmath,amsthm,amsfonts,latexsym}
\usepackage{mathrsfs}

\newtheorem{theorem}{Theorem}[section]
\newtheorem{lemma}[theorem]{Lemma}
\newtheorem{corollary}[theorem]{Corollary}

\newtheorem{proposition}[theorem]{Proposition}

\numberwithin{equation}{section}

\def \N {{\mathbb N}}

\def \Z {{\mathbb Z}}
\def \le{{\,\leqslant\,}}
\def \ge{{\,\geqslant\,}}

\def\pmod #1{({\rm mod}\ #1)}

\newcommand{\rank}{{\rm{rank}}}
\newcommand{\rankoff}{{\rm{rank}_{\rm{off}}}}

\begin{document}
\title[Quadratic forms and dense sets of primes]{Translation invariant quadratic forms and dense sets of primes}
\author{Lilu Zhao}
\address{School of Mathematics \\ Shandong University \\ Jinan  250100 \\ China}
\email{zhaolilu@sdu.edu.cn}

\begin{abstract}Let $f(x_1,\ldots,x_s)$ be a translation invariant indefinite quadratic
form of integer coefficients with $s\ge 10$. Let $\mathcal{A}\subseteq \mathcal{P}\cap \{1,2,\ldots,X\}$. Let $X$ be sufficiently large. Subject to a rank condition, we prove that there exist distinct primes $p_1,\ldots,p_s\in \mathcal{A}$ such that $f(p_1,\ldots,p_s)=0$ as soon as $|\mathcal{A}|\ge  \frac{X}{\log X} (\log\log X)^{-\frac{1}{80}}.$
\end{abstract}

\maketitle


{\let\thefootnote\relax}

{\let\thefootnote\relax}

{\let\thefootnote\relax\footnotetext{This work is supported by the NSFC grant 11922113.}}

\section{Introduction}

Roth's theorem \cite{R1} on arithmetic progressions of length three states that
\begin{align}\label{r3}\rho_3(X)\ll (\log\log X)^{-1},\end{align}
where for $k\ge 3$, we use $\rho_k(X)$ to denote the maximal density of subsets $\mathcal{A}$ in $\{1,2,\ldots,X\}$ satisfying that $\mathcal{A}$ contains no nontrivial arithmetic progressions of length $k$. Szemer\'{e}di \cite{Szk} proved $\rho_k(X)=o(1)$ for all $k$, and therefore confirmed a conjecture of Erd\"{o}s and Tur\'{a}n \cite{ET}.  In 2005 Green \cite{Green} established a remarkable analogue of Roth's theorem in primes, which states that any set containing a positive proportion of the primes contains a nontrivial $3$-term arithmetic progression.

In recent years, there has been much work on extensions of Roth's theorem to nonlinear equations. Browning and Prendiville \cite{BP} studied the diagonal quadratic equation
\begin{align}\label{BP}c_1x_1^2+c_2x_2^2+\cdots+c_sx_s^2=0\end{align}
subject to the condition $c_1+\cdots +c_{s}=0$, where the variables $x_1,\ldots,x_s$  are restricted
in $\mathcal{A}\subseteq \{1,2,\ldots,X\}$. It was proved in \cite{BP} that if $s\ge 5$ and the equation \eqref{BP} has only trivial solutions then
\begin{align} \label{BPdensity}|\mathcal{A}|/X \ll (\log \log \log X)^{-\frac{s}{2}+1+\varepsilon}.\end{align}
Chow \cite{Chow} considered the diagonal equation of degree $d$ over a subset of prime numbers, that is
\begin{align}\label{Chow}c_1x_1^d+c_2x_2^d+\cdots+c_sx_s^d=0,\end{align}
where $x_1,\ldots,x_s$  are restricted
in $\mathcal{A}\subseteq \mathcal{P}\cap\{1,2,\ldots,X\}$. Throughout this paper, we use $\mathcal{P}$ to denote the set of all prime numbers. It was proved in \cite{Chow}, if $s\ge d^2+1$ and the equation \eqref{Chow} has only trivial solutions, then
\begin{align} \label{Chowdensity}\frac{|\mathcal{A}|}{X/\log X} \ll (\log \log \log \log X)^{-\frac{s-2}{d}+\varepsilon}.\end{align}
One may also refer to recent impressive works \cite{CLP,MS} on diagonal equations of higher degree. Chow, Lindqvist and Prendiville \cite{CLP} considered Rado's type theorem over squares and higher powers. Matom\"aki and Shao \cite{MS} investigated the Waring-Goldbach problem in short intervals (see also \cite{MMS} for the linear case).


We study the quadratic equation
 \begin{align}\label{quadequation}f(x_1,\ldots,x_s)=0,\end{align}
where $f(x_1,\ldots,x_s)=\mathbf{x}M\mathbf{x}^{T}$ is a quadratic form with integral coefficients throughout. In other words,
\begin{align}\label{defM} M=\begin{pmatrix}a_{1,1}  & \cdots & a_{1,s}
\\ \vdots & \cdots & \vdots
\\ a_{s,1} & \cdots  & a_{s,s} \end{pmatrix}\end{align}
with $a_{i,j}=a_{j,i}\in\Z$ for all $1\le i< j\le s$. Liu \cite{Liu} initiated the investigation of prime solutions to \eqref{quadequation} when $s\ge 10$. Subject to a rank condition, Liu obtained the asymptotic formula for
\begin{align}\label{nuLiu}
\nu_f=\sum_{\substack{1<x_1,\ldots,x_s\le X\\ f(x_1,\ldots,x_s)=0  }}\Lambda(x_1)\cdots \Lambda(x_s),\end{align}
 where $\Lambda(\cdot)$ is the Von Mangoldt function. Motivated by the work of Liu \cite{Liu}, Keil \cite{Keil} introduced the off-diagonal rank of $M$
\begin{align}\label{defoff}\rankoff(M)
=\max\{r:\ r\in R\},\end{align}where
\begin{align*}R=\Big\{\rank(B):\ B=(a_{i_k,j_l})_{1\le k,l\le r}\
\textrm{ with }\
\{i_1,\ldots,i_r\}\cap\{j_1,\cdots,j_r\}=\emptyset\Big\}.\end{align*}
In other words, $\rankoff(M)$ is the maximal rank of a submatrix
in $M$, which does not contain any diagonal entries.
For a quadratic form $f(x_1,\ldots,x_s)=\mathbf{x}M\mathbf{x}^T$, we define the off-diagonal rank of $f$
\begin{align}\label{defoff2}\rankoff(f)=\rankoff(M).\end{align}
Essentially, Liu \cite{Liu} obtained the asymptotic formula for $\nu_f$ in \eqref{nuLiu} by assuming $\rankoff(f)\ge 5$.

Keil \cite{Keil,Keilrefine} considered the equation \eqref{quadequation} over dense sets of integers when $f$ is translation invariant, i.e. $\mathbf{1}M=\mathbf{0}$. We use $\mathbf{1}$ and $\mathbf{0}$ to denote $s$-dimensional vectors $(1,\ldots,1)\in \Z^s$ and $(0,\ldots,0)\in \Z^s$, respectively. Keil \cite{Keil} (see Theorem 2.2 in \cite{Keil}) proved that if $f$ is translation invariant with $\rankoff(f)\ge 5$ and the equation \eqref{quadequation} has only trivial solutions with variables restricted
in $\mathcal{A}\subseteq \{1,2,\ldots,X\}$ , then
\begin{align} \label{Keildensity}|\mathcal{A}|/X \ll (\log\log X)^{-c}\end{align}
for some absolute constant $c>0$. The density estimate \eqref{Keildensity} was improved in \cite{Keilrefine} (see Theorem 2.2 in \cite{Keilrefine}) to
\begin{align*}|\mathcal{A}|/X \ll (\log X)^{-c}.\end{align*}
Keil's above results were refined by Zhao \cite{Zhao} to a wide class of translation invariant quadratic forms in $9$ variables.

 The goal of this paper is to find nontrivial solutions to the translation invariant equation \eqref{quadequation}, where the variables are restricted in a subset of primes. The main result is the following.

\begin{theorem}\label{theorem1}Let $f(x_1,\ldots,x_s)$ be a translation invariant indefinite quadratic
form with $s\ge 10$. Suppose that $\rankoff(f)\ge 5$. Let $\mathcal{A}\subseteq \mathcal{P}\cap \{1,2,\ldots,X\}$. Suppose that there are no pairwise distinct primes  $p_1,\ldots,p_s\in \mathcal{A}$ such that $f(p_1,\ldots,p_s)=0$. Then we have
\begin{align}\label{thmdensity}\frac{|\mathcal{A}|}{X/\log X}\ll_f  (\log\log X)^{-\frac{1}{80}}.\end{align}
\end{theorem}

The proof of Theorem \ref{theorem1} involves several important methods in number theory, such as the Hardy-Littlewood (circle) method, the sieve method, Green's W-trick and Roth's method of density increment.

We apply the Hardy-Littlewood method to establish the asymptotic formula for the (weighted) number of solutions to \eqref{quadequation} with $x_1,\ldots,x_s\in \mathcal{A}$ when $\mathcal{A}$ has nice arithmetic distributions. Let
\begin{align}\frac{X}{2}\le Y_0<Y_0+Y\le X\ \ \ \textrm{ and }\ \ \ Y>\frac{X}{\log X}.\end{align}
Then we use $I$ to denote the interval
\begin{align}\label{defineintervalI}I=(Y_0,\,Y_0+Y].\end{align}
Let
\begin{align}\label{assumptionbW}0\le b< W\le \log X\ \  \textrm{ and }\ \ (b,W)=1.\end{align}
We introduce
\begin{align*}
\nu_f(b,W;I)=\sum_{\substack{x_1,\ldots,x_s\in I\\ f(\mathbf{x})=0 \\ \mathbf{x}\equiv b\pmod{W} }}\Lambda(x_1)\cdots \Lambda(x_s),\end{align*}
where $\mathbf{x}\equiv b\pmod{W}$ means $x_j\equiv b\pmod{W}$ for all $1\le j\le s$. We have the following result.
\begin{proposition}\label{propextendLiu}Let $f(x_1,\ldots,x_s)$ be an indefinite quadratic
form with $s\ge 10$. Suppose that $\rankoff(f)\ge 5$. Then we have
\begin{align*}\nu_f(b,W;I)=\mathfrak{S}^\ast_f(W,b) \mathfrak{K}_f(I) +O_f\big(Y^{s-2}W^2\phi(W)^{-s}(\log X)^{-5s}\big),\end{align*}
where the singular series $\mathfrak{S}^\ast_f(W,b)$ is defined in \eqref{definesingularseries}
and the singular integral $\mathfrak{K}_f(I)$ is defined in \eqref{definesingularint}.
\end{proposition}

Proposition \ref{propextendLiu} is a routine extension of Theorem 2.1 of Liu \cite{Liu} with variables in short intervals and arithmetic progressions. We have to do such an extension because it will be used to prove Theorem \ref{theorem1}. We can explain more on the singular series $\mathfrak{S}^\ast_f(W,b)$
and the singular integral $\mathfrak{K}_f(I)$ when $f$ is translation invariant.

\begin{corollary}\label{corollary13}Let $f(x_1,\ldots,x_s)$ be a translation invariant indefinite quadratic
form with $s\ge 10$. Suppose that $\rankoff(f)\ge 5$. Then we have
\begin{align}\label{asymnutran}\nu_f(b,W;I)=\frac{W^2Y^{s-2}}{\phi(W)^s}\mathfrak{S}_f(W) \mathfrak{I}_f  +O_f\big(Y^{s-2}W^2\phi(W)^{-s}(\log X)^{-5s}\big),\end{align}
where $\phi(\cdot )$ is Euler's totient function, $\mathfrak{S}_f(W)$ is defined in \eqref{definedensityprod} and  $\mathfrak{I}_f$ is defined in \eqref{definesingularinttran}.
Moreover, there exists a positive number $C_f$ (independent of $W$), such that
\begin{align*}\mathfrak{S}_f(W)\mathfrak{I}_f>C_f.\end{align*}
\end{corollary}

Corollary \ref{corollary13} yields Theorem \ref{theorem1} in the special case when $\mathcal{A}=\{p\in I:\, p\equiv b\pmod{W}\}$. Now we turn to an arbitrary set $\mathcal{A}\subseteq \mathcal{P}\cap \{1,2,\ldots,X\}$. Following Green \cite{Green}, we consider  \begin{align}\label{defineAp}\mathcal{A}'=\{y:\ Wy+b\in \mathcal{A}\cap I\}.\end{align}
It is clear the set $\mathcal{A}'$ is closely related to primes $x\in \mathcal{A}\cap I$ in the arithmetic progression $b\pmod{W}$. As usual in this topic, we use the letter $W$ to express we shall apply the $W$-trick. We highlight a difference in our proof. In \cite{BP,Chow,Green}, one may need to consider the congruence modulo $\prod_{p\le W}p$, while we consider the congruence $x\equiv b\pmod{W}$. Note that $\prod_{p\le W}p\approx e^W$. This is perhaps an apparent reason why we save a logarithmic symbol comparing to \eqref{BPdensity} of Browning and Prendiville. In order to deal with an arbitrary set $\mathcal{A}'$, we have to study the restriction theory. We introduce the function
\begin{align}\label{defineLambdabWI}\Lambda_{b,W;I}(x)=\begin{cases}\Lambda(Wx+b) \ \ & \textrm{ if }\ Wx+b\in I,
\\ 0\ \ &  \textrm{ otherwise}.\end{cases}\end{align}
For any $i\ge 1$, let $\{\lambda_{i}(x)\}_{x=1}^\infty$ be a sequence satisfying
\begin{align}\label{definelambdai}|\lambda_{i}(x)|\le \Lambda_{b,W;I}(x).\end{align}
Note that the characteristic function $1_{\mathcal{A}'}$ satisfies
$$|1_{\mathcal{A}'}(x)|\le \Lambda_{b,W;I}(x).$$
 We assume there exists $0<\delta\le 2$ such that
\begin{align}\label{assumptiononlambdai}\sum_{x}|\lambda_{i}(x)|\le \frac{\delta Y}{\phi(W)}\end{align}
for all $i\ge 1$. We consider the exponential sum $S(\alpha):=S(\alpha;\lambda_1,\ldots,\lambda_s)$ in the following
\begin{align}\label{defSalpha}S(\alpha)=
\sum_{x_1,\ldots,x_s}\lambda_1(x_1)\ldots \lambda_s(x_s)
e\big(\alpha f(x_1,\ldots,x_s)\big).\end{align}
We shall establish a restriction estimate for $S(\alpha)$ over minor arcs. We define
the major arcs
\begin{align}\label{defineMQ}\mathfrak{M}(Q)=\bigcup_{1\le q\le
Q}\bigcup_{\substack{a=1
 \\ (a,q)=1}}^q\mathfrak{M}(q,a;Q),\end{align}
where the intervals $\mathfrak{M}(q,a;Q)$ are
\begin{align*}\mathfrak{M}(q,a;Q)=\Big\{\alpha:\ \big|\alpha-\frac{a}{q}\big|\le \frac{Q}{q(Y/W)^2}\Big\}.\end{align*}
The intervals $\mathfrak{M}(q,a;Q)$ are pairwise disjoint for $1\le
a\le q\le Q$ and $(a,q)=1$ provided that $2Q<  Y/W$. Then for $Q< \frac{Y}{2W}$, we define the minor arcs
\begin{align}\label{definemQ}\mathfrak{m}(Q)=[(Y/W)^{-1},1+(Y/W)^{-1}]\setminus\mathfrak{M}(Q).\end{align}
Our restriction estimate over minor arcs is as follows.
\begin{proposition}\label{proprestriction}Let $f(x_1,\ldots,x_s)$ be an indefinite quadratic
form with $s\ge 10$. Suppose that $\rankoff(f)\ge 5$. Let $S(\alpha)$ be defined in \eqref{defSalpha}. Suppose that $Q\le \log X$. Then we have
\begin{align}\label{restrictionestimate}\int_{\mathfrak{m}(Q)}|S(\alpha)|d\alpha\ll_f \frac{\delta ^{s-10} W^2Y^{s-2}}{\phi(W)^{s}}Q^{-\frac{10}{21}},\end{align}
where the implied constant depends only on the form $f$.
\end{proposition}

Recently, there are several interesting articles on the application of the circle method to the restriction estimate. We refer readers to \cite{HH,Wooley} for the investigation in this topic.

  One can easily prove a weaker version of Proposition \ref{proprestriction} with  an extra factor $(\log X)^{10}$ on the right hand side of \eqref{restrictionestimate}, by using the result from \cite{Liu}. However, our proof of Theorem \ref{theorem1} would fail even if there were an extra factor $\log\log X$. It would also fail if there were an extra factor $W^{0.01}$. Note that the upper bound in \eqref{restrictionestimate} almost coincides with the right order of $\nu_f(b,W;I)$ in \eqref{asymnutran} up to a constant. In additive prime number theory, it is well-known that the sieve method can be applied to obtain an upper bound, which is a constant multiple of the right order. Therefore, in the proof of Proposition \ref{proprestriction}, we not only benefit from the $W$-trick, but also take advantage of the sieve theory. The combination of the circle method and sieve method has many applications in the Waring-Goldbach problem, and one may refer to Br\"udern \cite{Brudern} and Kawada-Wooley \cite{KW}. The result in this paper can be viewed as a new example, in which the circle method and sieve method work together.

We point out a new feature in the proof here. In order to capture solutions in dense subsets of primes, in previously works (see \cite{Chow,Green}, for example), one may use the transference principle. Since Proposition \ref{proprestriction} provides an acceptable restriction estimate over minor arcs, we can avoid the use of transference principle and instead we can apply Roth's method directly. Therefore, with the asymptotic information in Corollary \ref{corollary13} and the restriction estimate in Proposition \ref{proprestriction}, we are able to apply Roth's argument of density increment to complete the proof of Theorem \ref{theorem1}.

\vskip3mm

In Section 2, we prepare some technical lemmas to explain the singular series and singular integral. We shall prove Proposition \ref{propextendLiu} and its corollary in Section 3. We start to prove Proposition \ref{proprestriction} in Section 5, and we shall finish it in Section 5. Finally, we complete the proof of Theorem \ref{theorem1} in Section 6.
\vskip3mm

As usual, we write $e(z)$ for $e^{2\pi iz}$. We assume
that $X$ is sufficiently large. We use $\ll$ and
$\gg$ to denote Vinogradov's well-known notations. The implied constant may depend on $f$. Denote by $\phi(q)$ Euler's totient function, and $\tau(q)$ the divisor function. For a finite subset $A\subseteq \N$, we denote by $|A|$ the cardinality of $A$, while for an interval $J$, we use
$|J|$ to denote the length of $J$.


We use bold face letters to denote vectors whose dimensions are clear from the context. For $\mathbf{x}=(x_1,\ldots,x_s)\in \Z^s$ and a function $\xi$, we use $\xi(\mathbf{x})$ to denote the product $\prod_{i=1}^s \xi(x_i)$. We use $\mathfrak{A}(\mathbf{x})$ to indicate that
$\mathfrak{A}(x_i)$ holds for all $i$. The meaning will
be clear from the context. For example,
the congruence $\mathbf{x}\equiv \mathbf{y}\pmod{q}$ means $x_i\equiv y_i\pmod{q}$ for all $i$, while for $b\in \Z$, we use $\mathbf{x}\equiv b\pmod{q}$ to indicate $x_i\equiv b\pmod{q}$ for all $i$.

\vskip3mm
\section{Preparations}

Throughout this paper, we assume $\rankoff(f)\ge 5$, although most results in this section can be proved subject to the weaker condition $\rank(M)\ge 5$.
We begin with the multiple Gauss sum
\begin{align}\label{definegausssum}S_{W,b}^\ast(q,a)=\sum_{\substack{1\le \mathbf{c}\le qW\\ (\mathbf{c},q)=1
\\ \mathbf{c}\equiv b\pmod{W}}}e\big(
\frac{af(\mathbf{c})}{q}\big).\end{align}
When $W=1$, we write
\begin{align}\label{definegausssum2}S^\ast(q,a)=S_{1,0}^\ast(q,a).\end{align}
We also define
\begin{align}\label{definegausssum3}S(q,a)=\sum_{\substack{1\le \mathbf{c}\le q}}e\big(
\frac{af(\mathbf{c})}{q}\big).\end{align}

\begin{lemma}\label{lemmaGaussmul}Suppose that $(q_1W_1,q_2W_2)=1$. Then we have
\begin{align*}S_{W_1W_2,b}^\ast(q_1q_2,a)=S_{W_1,b}^\ast(q_1,a\overline{q_2})S_{W_2,b}^\ast(q_2,a\overline{q_1}),\end{align*}
where $\overline{x}$ in $e\big(\frac{s\overline{x}}{r})$ denotes the inverse of $x$ modulo $r$.
\end{lemma}
\begin{proof} We write
$$\mathbf{c}=\mathbf{c}_1\overline{q_2W_2}\,q_2W_2+\mathbf{c}_2\overline{q_1W_1}\,q_1W_1,$$
 where $\overline{q_1W_1}$ means the inverse of $q_1W_1$ modulo $q_2W_2$, and $\overline{q_2W_2}$ means the inverse of $q_2W_2$ modulo $q_1W_1$. Then the congruence
$\mathbf{c}\equiv b\pmod{W_1W_2}$ is equivalent to
\begin{align*}\mathbf{c}_1\equiv b\pmod{W_1} \ \ \textrm{ and } \ \ \mathbf{c}_2\equiv b\pmod{W_2}.\end{align*}
 We also have
 \begin{align*}e\big(
\frac{af(\mathbf{c})}{q_1q_2}\big)=e\big(
\frac{af(\mathbf{c}_1)\overline{q_2}}{q_1}\big)e\big(
\frac{af(\mathbf{c}_2)\overline{q_1}}{q_2}\big).\end{align*}
The desired result can be obtained by changing variables as above.\end{proof}

\begin{lemma}\label{lemmaGausspower}Let $p$ be a prime and $(ab,p)=1$. Let $t\ge 0$ and $k\ge 0$. Then we have
\begin{align}\label{gausssumboundpower}S_{p^t,b}^\ast(p^k,a)\ll p^{sk-\frac{5}{2}k+5t+\varepsilon}.\end{align}
If $f$ is translation invariant, then we have
\begin{align*}S_{p^t,b}^\ast(p^k,a)=\begin{cases}S^\ast(p^k,a) \ \ \ &\textrm{ if }\ t=0,
\\ p^{sk} \ \ \ &\textrm{ if }\ t\ge 1\ \textrm{ and }\ 0\le k\le 2t,
\\ p^{2st}S(p^{k-2t},a) \ \ \ &\textrm{ if }\ t\ge 1\ \textrm{ and }\ k>2t.\end{cases}\end{align*}
\end{lemma}
\begin{proof}When $t=0$, we have $S_{p^t,b}^\ast(p^k,a)=S^\ast(p^k,a)$ and its upper bound has been obtained in \cite{Liu} (see the proof of Lemma 5.1 in \cite{Liu}).

We next consider the case $t\ge 1$. We deduce by changing variables $\mathbf{c}=b\mathbf{1}+p^t\mathbf{x}$ that
\begin{align*}S_{p^t,b}^\ast(p^k,a)=e\big(
\frac{ab^2f(\mathbf{\mathbf{1}})}{p^{k}}\big)\sum_{\substack{1\le \mathbf{x}\le p^{k}
}}e\big(
\frac{2abp^t\mathbf{1}M\mathbf{x}^{T}+af(\mathbf{\mathbf{x}})p^{2t}}{p^{k}}\big).\end{align*}
For $k>2t$, by the standard difference argument, we can obtain
\begin{align}\label{ineqk2t}S_{p^t,b}^\ast(p^k,a)\ll p^{sk-\frac{5}{2}k+5t+\varepsilon}.\end{align}
For $k\le 2t$, the estimate \eqref{ineqk2t} holds trivially. This completes the proof of \eqref{gausssumboundpower}.

Now we further assume that $f$ is translation invariant. Then
\begin{align*}S_{p^t,b}^\ast(p^k,a)=\sum_{\substack{1\le \mathbf{x}\le p^{k}
}}e\big(
\frac{af(\mathbf{x})p^{2t}}{p^{k}}\big)=\begin{cases}p^{sk} \ \ &\textrm{ if }\ \ 0\le k\le 2t,
\\ p^{2st}S(p^{k-2t},a)\ \ &\textrm{ if }\ \ k>2t.\end{cases}\end{align*}
This completes the proof.\end{proof}

We define
\begin{align}\label{defineBWb}B_{W,b}^\ast(q)=\frac{1}{\phi(qW)^s}\sum_{\substack{a=1 \\ (a,q)=1}}^qS_{W,b}^\ast(q,a),\end{align}
where $S_{W,b}^\ast(q,a)$ is given in \eqref{definegausssum}. Then we write
\begin{align}B^\ast(q)=B_{1,0}^\ast(q),\end{align}
and define
\begin{align}B(q)=\frac{1}{q^s}\sum_{\substack{a=1 \\ (a,q)=1}}^qS(q,a).\end{align}

The following result is Lemma 5.1 in \cite{Liu}.
\begin{lemma}\label{lemmaboundBq}Let $B^\ast(q)$ and $B(q)$ be defined as above. Then we have
\begin{align*}B^\ast(q)\ll q^{-\frac{3}{2}+\varepsilon}\  \textrm{ and }\  B(q)\ll q^{-\frac{3}{2}+\varepsilon}.\end{align*}
\end{lemma}
We point out the statement of Lemma 5.1 in \cite{Liu} provides the estimate for $B^\ast(q)$ only, while the proof works for $B(q)$. Now we consider $B_{W,b}^\ast(q)$. Lemma \ref{lemmaGaussmul} implies the following result.
\begin{lemma}\label{lemmaBWbmulti}Suppose that $(q_1W_1,q_2W_2)=1$. We have
\begin{align*}B_{W_1W_2,b}^\ast(q_1q_2)=B_{W_1,b}^\ast(q_1)B_{W_2,b}^\ast(q_2).\end{align*}
\end{lemma}
Next result is a crude upper bound of $B_{W,b}^\ast(q)$.
\begin{lemma}Suppose that $(b,W)=1$. Then we have
\begin{align}\label{BWbbound}B_{W,b}^\ast(q)\ll \frac{W^5}{\phi(W)^s}q^{-\frac{3}{2}+\varepsilon}.\end{align}
\end{lemma}
\begin{proof}By \eqref{gausssumboundpower}, we obtain
  \eqref{BWbbound} in the case $q=p^k$ and $W=p^t$ with $p$ a prime. Then the estimate \eqref{BWbbound} for general $q$
 follows from Lemma \ref{lemmaBWbmulti}. This completes the proof.
\end{proof}
With the estimate \eqref{BWbbound}, we are able to introduce the singular series
\begin{align}\label{definesingularseries}\mathfrak{S}_{f}^\ast(W,b)=\sum_{q=1}^\infty
B_{W,b}^\ast(q),\end{align}
where $B_{W,b}^\ast(q)$ is given in \eqref{defineBWb}.
For each prime $p$, we introduce local densities
\begin{align}\label{definelocaldensity1}\sigma_p^\ast=\sum_{k=0}^\infty
B^\ast(p^k)\end{align}
and
\begin{align}\label{definelocaldensity2}\sigma_p=\sum_{k=0}^\infty
B(p^k).\end{align}
Now we introduce $\mathfrak{S}(W)$, a product of local densities, defined to be
\begin{align}\label{definedensityprod}\mathfrak{S}(W)=\Big(\prod_{p|W}\sigma_p\Big) \Big(\prod_{p\nmid W}\sigma_p^\ast\Big).\end{align}

In order to understand local densities, we point out
\begin{align}\label{limitsigmaast}\sigma_p^\ast=\lim_{k\rightarrow\infty}\,\frac{p^k}{\phi(p^k)^{s}}|\{1\le \mathbf{x}\le p^k: \ (\mathbf{x},p)=1\ \textrm{ and } \
f(\mathbf{x})\equiv 0\pmod{p^k}\}|\end{align}
and
\begin{align}\label{limitsigma}\sigma_p=\lim_{k\rightarrow \infty}\,\frac{1}{p^{k(s-1)}}|\{1\le \mathbf{x}\le p^k:  \
f(\mathbf{x})\equiv 0\pmod{p^k}\}|.\end{align}

\begin{lemma}\label{lemmatwodensity}Suppose that $f$ is translation invariant. Then we have
\begin{align*}\sigma_p^\ast \ge \frac{p^{2}}{\phi(p)^s}\sigma_p.\end{align*}
\end{lemma}
\begin{proof}Let $k\ge 1$. We consider the solutions to $f(\mathbf{x})\equiv 0\pmod{p^{k+2}}$ with $\mathbf{x}$ in the form
$$\mathbf{x}=\mathbf{1}+p\mathbf{y}\ (1\le \mathbf{y}\le p^{k+1}),$$
and deduce that
\begin{align*}&|\{1\le \mathbf{x}\le p^{k+2}: \ (\mathbf{x},p)=1\ \textrm{ and } \
F(\mathbf{x})\equiv 0\pmod{p^{k+2}}\}|
\\ \ge\, & |\{1\le \mathbf{y}\le p^{k+1}: \
F(\mathbf{y})\equiv 0\pmod{p^{k}}\}|
\\ =\ & p^{s}|\{1\le \mathbf{y}\le p^{k}: \
F(\mathbf{y})\equiv 0\pmod{p^{k}}\}|.\end{align*}
Then it is easy to see
\begin{align*}&\frac{p^{k+2}}{\phi(p^{k+2})^{s}}|\{1\le \mathbf{x}\le p^{k+2}: \ (\mathbf{x},p)=1\ \textrm{ and } \
F(\mathbf{x})\equiv 0\pmod{p^{k+2}}\}|
\\ \ge &\frac{p^{k+2+s}}{\phi(p^{k+2})^{s}}|\{1\le \mathbf{y}\le p^{k}: \
F(\mathbf{y})\equiv 0\pmod{p^{k}}\}|
\\ = &  \frac{p^{2}}{\phi(p)^s} \cdot \frac{1}{p^{(s-1)k}}|\{1\le \mathbf{y}\le p^{k}: \
F(\mathbf{y})\equiv 0\pmod{p^{k}}\}|.\end{align*}
We conclude from \eqref{limitsigmaast} and \eqref{limitsigma} that $\sigma_p^\ast\ge \frac{p^{2}}{\phi(p)^s}\sigma_p$.
\end{proof}

\begin{lemma}\label{lemmasingularseries}Let $(b,W)=1$. Suppose that $f$ is translation invariant. Then we have\begin{align}\label{representsingularseries}\mathfrak{S}_{f}^\ast(W,b)=\frac{W^2}{\phi(W)^s}\mathfrak{S}(W),\end{align}
where $\mathfrak{S}(W)$ is given in \eqref{definedensityprod}. Moreover, there exists a positive number $C(f)$, independent of $W$, such that
\begin{align}\label{densityprodlower}\mathfrak{S}(W)>C(f)>0.\end{align}
\end{lemma}
\begin{proof}By Lemma \ref{lemmaBWbmulti}, we have
\begin{align*}\mathfrak{S}_{f}^\ast(W,b)=\prod_{p^t\|W}\sum_{k=0}^\infty B^\ast_{p^t,b}(p^k),\end{align*}
where the product is taken over all primes, that is, $t$ maybe equals zero.

In the case $p\nmid W$ (which implies $t=0$), we apply Lemma \ref{lemmaGausspower} to deduce that
\begin{align*}\sum_{k=0}^\infty B^\ast_{p^t,b}(p^k)=\sigma_p^\ast.\end{align*}
Therefore, in order to prove \eqref{representsingularseries}, it remains to show for $p|W$ (which implies $t\ge 1$), we have
\begin{align*}\sum_{k=0}^\infty B^\ast_{p^t,b}(p^k)=\frac{p^{2t}}{\phi(p^t)^s}\sigma_p.\end{align*}
We apply Lemma \ref{lemmaGausspower} again to obtain
\begin{align*}B_{p^t,b}(p^k)=\begin{cases}
 \phi(p^k)\phi(p^{t})^{-s} \ \ \ &\textrm{ if }\ t\ge 1\ \textrm{ and }\ 0\le k\le 2t,
\\ p^{2t}\phi(p^{t})^{-s}B(p^{k-2t}) \ \ \ &\textrm{ if }\ t\ge 1\ \textrm{ and }\ k>2t.\end{cases}\end{align*}
Then we deduce that
\begin{align*}\sum_{k=0}^\infty B^\ast_{p^t,b}(p^k)=\, &\sum_{k=0}^{2t}
\frac{\phi(p^k)}{\phi(p^{t})^{s}} +\sum_{k=2t+1}^{\infty}\frac{p^{2t}}{\phi(p^{t})^{s}}B(p^{k-2t})
\\ =\, & \frac{p^{2t}}{\phi(p^{t})^{s}}+\frac{p^{2t}}{\phi(p^{t})^{s}}\sum_{k=2t+1}^{\infty}B(p^{k-2t})=\frac{p^{2t}}{\phi(p^t)^s}\sigma_p.\end{align*}
This completes the proof of \eqref{representsingularseries}.

By Lemma \ref{lemmaboundBq}, we have
\begin{align*}\sigma_p^\ast =1+O(p^{-\frac{3}{2}+\varepsilon}) \ \ \textrm{ and } \ \ \sigma_p=1+O(p^{-\frac{3}{2}+\varepsilon}).\end{align*}
Therefore, there exists a natural number $N_f>0$ such that
\begin{align*}\Big(\prod_{\substack{p|W
\\ p\ge N_f}}\sigma_p\Big) \Big(\prod_{\substack{p\nmid W\\ p\ge N_f}}\sigma_p^\ast\Big)>C_1(f)\end{align*}
for some $C_1(f)>0$. Now in order to prove \eqref{densityprodlower}, we need to verify
$\sigma_p>0$ and $\sigma_p^\ast>0$ for all $p< N_f$. By Lemma \ref{lemmatwodensity}, we only need to show $\sigma_p>0$ for all $p< N_f$.
It is well-known that if $f$ is a quadratic form with $\rank(M)\ge 5$, then $\sigma_p>0$ all prime $p$. The proof is completed.
\end{proof}

The study of the singular integral is easier because the singular integral in our paper in the same as
that in the corresponding problem with integral variables.  We define
\begin{align}\label{definecalKint}\mathcal{K}(\beta;I)=\int_{I^s} e\big(\beta f(\mathbf{x})\big)d\mathbf{x}.\end{align}
Subject to the condition $\rank(M)\ge 5$, we have
\begin{align}\label{boundcalK}\mathcal{K}(\beta;I)\ll Y^s(1+Y^2|\beta|)^{-2}.\end{align}
Then we define the singular integral
\begin{align}\label{definesingularint}\mathfrak{K}_f(I)=\int_{-\infty}^{+\infty} \mathcal{K}(\beta;I)d\beta.\end{align}
We also define
\begin{align}\label{definesingularinttran}\mathfrak{I}_f=\int_{-\infty}^{+\infty}\Big( \int_{(0,1]^s} e\big(\beta f(\mathbf{x})\big)d\mathbf{x}\Big)d\beta.\end{align}

The following lemma can be proved easily by changing variables.
\begin{lemma}\label{lemmasingularint}
Suppose that $f$ is translation invariant. Then we have
\begin{align}\mathfrak{K}_f(I)=Y^{s-2}\mathfrak{I}_f.\end{align}
Moreover, if $f$ is indefinite then $\mathfrak{I}_f>0$.
\end{lemma}

We define
\begin{align}\label{defineTWq}R_{W,b}^\ast(q,a,h)=\sum_{\substack{1\le c\le qW
\\ (c,q)=1 \\ c\equiv b\pmod{W}}}e\big(\frac{ach}{qW}\big).\end{align}
\begin{lemma}\label{lemmaTmul}Suppose that $(W_1q_1,W_2q_2)=1$. Then we have
\begin{align*}R^\ast_{W_1W_2,b}(q_1q_2,a,h)=R^\ast_{W_1,b}(q_1,a\overline{q_2W_2},h)R^\ast_{W_2,b}(q_2,a\overline{q_1W_1},h).\end{align*}
\end{lemma}
\begin{proof}We can confirm the desired conclusion by a similar argument in the proof of Lemma \ref{lemmaGaussmul}.  We omit the details.
\end{proof}

\begin{lemma}\label{lemmaTpower}Suppose that $p$ is a prime, $k\ge 0$ and $t\ge 0$. Suppose that $(ab,p)=1$. Then we have
\begin{align}\label{Testimate}|R^\ast_{p^t,b}(p^k,a,h)|\le (h,p^{k}).\end{align}
\end{lemma}
\begin{proof}We first consider $t=0$, and in this case we have
\begin{align*}R^\ast_{p^t,b}(p^k,a,h)=\sum_{\substack{1\le c\le p^k
\\ (c,p^k)=1 }}e\big(\frac{ach}{p^k}\big).\end{align*}
This is a Ramanujan sum, and we obtain \eqref{Testimate}.
If $t\ge 1$, then we deduce by changing variables $c=b+up^t$ that
\begin{align*}R^\ast_{p^t,b}(p^k,a,h)=\sum_{\substack{1\le u\le p^{k}
 }}e\big(\frac{a(b+up^t)h}{p^{k+t}}\big)=
 e\big(\frac{abh}{p^{k+t}}\big)\sum_{\substack{1\le u\le p^{k}
 }}e\big(\frac{auh}{p^{k}}\big).\end{align*}
Thus, we have $R^\ast_{p^t,b}(p^k,a,h)= e\big(\frac{abh}{p^{k+t}}\big)p^{k}$ or $0$ according to $p^{k}|h$ or not. In particular, the inequality \eqref{Testimate} holds. This completes the proof.
\end{proof}

\begin{lemma}\label{lemmaboundTWq}Suppose that $(W,b)=(q,a)=1$. Then we have
\begin{align*}|R^\ast_{W,b}(q,a,h)|\le (h,q).\end{align*}
\end{lemma}
\begin{proof}This follows from Lemma \ref{lemmaTmul} and Lemma \ref{lemmaTpower}.\end{proof}

\vskip3mm

\section{Proofs of Proposition \ref{propextendLiu} and Corollary \ref{corollary13}}

We define
\begin{align}\label{defTalpha}T(\alpha;\xi_1,\ldots,\xi_s)=
\sum_{\frac{X}{2}<x_1,\ldots,x_s\le X}\xi_1(x_1)\ldots \xi_s(x_s)
e\big(\alpha f(x_1,\ldots,x_s)\big)\end{align}
with $\xi_i(x)$ satisfying
\begin{align}\label{definexi}\sum_{\frac{X}{2}<x\le X}|\xi_i(x)|^2 \ll X\log X \ \textrm{ for } 1\le i\le s.\end{align}

The following result is essentially Lemma 3.7 of Liu \cite{Liu}.
\begin{lemma}Let $T(\alpha;\xi_1,\ldots,\xi_s)$ be defined in \eqref{defTalpha}. Suppose that $\alpha=\frac{a}{q}+\beta$ with $(a,q)=1$ and $|\beta|\le \frac{1}{q^2}$. Then we have
\begin{align}\label{Txi}T(\alpha;\xi_1,\ldots,\xi_s)\ll (X\log X)^{s}\big(\frac{1}{X}+\frac{1}{q(1+X^2|\beta|)}+\frac{q(1+X^2|\beta|)}{X^2}\big)^{\frac{5}{2}}.\end{align}\end{lemma}

 This result is useful when $\alpha \in \mathfrak{m}(Q)$ with
 $Q$ greater than a large power of $\log X$, say $Q\ge (\log X)^{20s}$. Thus we define
 \begin{align}\label{defineQ0}Q_0= (\log X)^{20s}.\end{align}

Similarly to Lemma 4.1 of Liu \cite{Liu}, by using \eqref{Txi}, we can prove the following result.
 \begin{lemma}\label{lemmaliu1}Let $T(\alpha;\xi_1,\ldots,\xi_s)$ be defined in \eqref{defTalpha}. Let $Q_0$ be given in \eqref{defineQ0}. Then we have
\begin{align}\label{boundTminor}\int_{\mathfrak{m}(Q_0)}|T(\alpha;\xi_1,\ldots,\xi_s)|d\alpha \ll X^{s-2}(\log X)^{-8s}.\end{align}\end{lemma}

On recalling $\lambda_i$ given in \eqref{definelambdai} and choosing $\xi_i=\lambda_i$ , we have the following result from \eqref{boundTminor}.
 \begin{lemma}\label{lemmaliu2}Let $S(\alpha)$ be defined in \eqref{defSalpha}. Let $Q_0$ be given in \eqref{defineQ0}. Then we have
\begin{align}\label{boundminor}\int_{\mathfrak{m}(Q_0)}|S(\alpha)|d\alpha \ll Y^{s-2}(\log X)^{-6s}.\end{align}\end{lemma}

Lemma \ref{lemmaliu1} is valid for arbitrary sequences $\xi_{i}$ given in \eqref{definexi}. Then in order to deal with the contribution from the major arcs $\mathfrak{M}(Q_0)$, we need arithmetic theory on the distribution of a sequence, such as the Siegel-Walfisz theorem. However, we do not have such a strong distribution theorem if we consider an arbitrary (dense) sequence of primes.

From now on, we write $T(\alpha)=T(\alpha;\xi_1,\ldots,\xi_s)$ with
$$\xi_1=\cdots=\xi_s=\chi_{b,W;I}\cdot \Lambda,$$
where $\chi_{b,W;I}$ the characteristic function for $\{x\in I:\ x\equiv b\pmod{W}\}$. In particular, we have
\begin{align}\label{circle1}\nu_f(b,W;I)=\int_0^1T(\alpha)d\alpha.\end{align}

\vskip3mm

\noindent {\it Proof of Proposition \ref{propextendLiu}}.
Suppose that $\alpha=a/q+\beta$ with
$$1\le a\le q\le Q_0, \ \ (a,q)=1\ \ \textrm{ and }\ \ |\beta|\le \frac{Q_0W^2}{qY^2}.$$
We introduce congruence conditions to deduce that
\begin{align*}T(\alpha)=\sum_{\substack{1\le \mathbf{c}\le qW\\ (\mathbf{c},q)=1 \\ \mathbf{c}\equiv b\pmod{W}}}e\big(\frac{af(\mathbf{c})}{q}\big)\sum_{\substack{\mathbf{x}\in I^s
\\ \mathbf{x}\equiv \mathbf{c}\pmod{qW}}}\Lambda(\mathbf{x})e\big(\beta f(\mathbf{x})\big)+O(Y^{s-1}\log X).\end{align*}
By the standard application of the Siegel-Walfisz theorem and the partial summation formula, we can establish
\begin{align*}\sum_{\substack{\mathbf{x}\in I^s
\\ \mathbf{x}\equiv \mathbf{c}\pmod{qW}}}\Lambda(\mathbf{x})e\big(\beta f(\mathbf{x})\big)=\frac{1}{\phi(qW)^s}\int_{I^s} e\big(\beta f(\mathbf{y})\big)d\mathbf{y}+O(Y^s(\log X)^{-800s}).\end{align*}
 Then we conclude from above
\begin{align}\label{Tasymp}T(\alpha)=
\frac{1}{\phi(qW)^s}S^\ast_{W,b}(q,a)\mathcal{K}(\beta;I)+O(Y^s(\log X)^{-80s}),\end{align}
where $S^\ast_{W,b}(q,a)$ and $\mathcal{K}(\beta;I)$ are given in \eqref{definegausssum} and \eqref{definecalKint}, respectively.

 By the definition of $\mathfrak{M}(Q_0)$ in \eqref{defineMQ}, we have
\begin{align*}\int_{\mathfrak{M}(Q_0)}T(\alpha)d\alpha=\sum_{q\le Q_0}\sum_{\substack{a=1 \\(a,q)=1}}^q \int_{|\beta|\le \frac{Q_0W^2}{qY^2}}
T(\frac{a}{q}+\beta)d\beta.\end{align*}
We deduce by \eqref{Tasymp} that
\begin{align*}\int_{\mathfrak{M}(Q_0)}T(\alpha)d\alpha=&\sum_{q\le Q_0}\frac{1}{\phi(qW)^s}\Big(\sum_{\substack{a=1 \\(a,q)=1}}^qS^\ast_{W,b}(q,a)\Big) \int_{|\beta|\le \frac{Q_0W^2}{qY^2}}\mathcal{K}(\beta ;I)d\beta
\\ &\ +O(Y^{s-2}W^2(\log X)^{-40s})\end{align*}
and thus
\begin{align}\label{intSalpha1}\int_{\mathfrak{M}(Q_0)}T(\alpha)d\alpha=\sum_{q\le Q_0}B^\ast_{W,b}(q) \int_{|\beta|\le \frac{Q_0W^2}{qY^2}}
\mathcal{K}\big(\beta;I\big)d\beta +O(Y^{s-2}W^2(\log X)^{-40s}).\end{align}

We deduce from \eqref{boundcalK} that
\begin{align}\label{intcalK}\int_{|\beta|\le \frac{Q_0W^2}{qY^2}}\mathcal{K}\big(\beta;I\big)d\beta=&\int_{-\infty}^{+\infty}\mathcal{K}\big(\beta;I\big)d\beta+O(Y^{s-2}qQ_0^{-1}W^{-2})\notag
\\ =& \mathfrak{K}_f(I)+O(Y^{s-2}qQ_0^{-1}W^{-2}).\end{align}
Then by \eqref{intSalpha1} and \eqref{intcalK}, we obtain
\begin{align*}\int_{\mathfrak{M}(Q_0)}T(\alpha)d\alpha=&\mathfrak{K}_f(I)\sum_{q\le Q_0}B^\ast_{W,b}(q)
+Y^{s-2}Q_0^{-1}W^{-2}
\sum_{q\le Q_0}O\big(q|B^\ast_{W,b}(q)| \big)
\\  & \ +O(Y^{s-2}W^2(\log X)^{-40s}) .\end{align*}
We conclude from \eqref{BWbbound} that
\begin{align*}\int_{\mathfrak{M}(Q_0)}T(\alpha)d\alpha=&\mathfrak{K}_f(I)\mathfrak{S}_f^\ast(W,b)
+O\big(Y^{s-2}W^5\phi(W)^{-s}Q_0^{-\frac{1}{2}+\varepsilon}\big)
\\  & \ +O(Y^{s-2}W^2(\log X)^{-40s}) .\end{align*}
In particular, we have
\begin{align}\label{boundmajor}\int_{\mathfrak{M}(Q_0)}T(\alpha)d\alpha=&\mathfrak{K}_f(I)\mathfrak{S}_f^\ast(W,b)
+O\big(Y^{s-2}W^2\phi(W)^{-s}(\log X)^{-8s}\big).\end{align}

Now we combine \eqref{boundTminor} and \eqref{boundmajor} to obtain
\begin{align*}\int_{0}^1T(\alpha)d\alpha=&\mathfrak{K}_f(I)\mathfrak{S}_f^\ast(W,b)
+O\big(Y^{s-2}W^2\phi(W)^{-s}(\log X)^{-5s}\big).\end{align*}
In view of \eqref{circle1}, this completes the proof of Proposition \ref{propextendLiu}.

\vskip3mm

\noindent {\it Proof of Corollary \ref{corollary13}}. When $f$ is translation invariant, the asymptotic formula \eqref{asymnutran} follows from Proposition \ref{propextendLiu} together with \eqref{representsingularseries} and Lemma \ref{lemmasingularint}. Moreover by Lemma \ref{lemmasingularseries} and Lemma \ref{lemmasingularint}, $\mathfrak{S}_f(W)\mathfrak{I}_f$ has a positive lower bound independent of $W$. This completes the proof of Corollary \ref{corollary13}.

\vskip3mm

\section{Initial step for restriction estimate: $W$-trick}

Since $\rank_{\rm off}(f)\ge 5$, without loss of generality, we assume $\rank(M_0)=5$, where
$M_0=(a_{i,j+5})_{1\le i,j\le 5}$, that is
\begin{align} \label{defineM0}M_0=\begin{pmatrix}a_{1,6}  & \cdots & a_{1,10}
\\ \vdots & \cdots & \vdots
\\ a_{5,6} & \cdots  & a_{5,10} \end{pmatrix}.\end{align}
For $\mathbf{x}=(x_1,\ldots,x_{5})\in \Z^5$ and $\mathbf{y}=(y_1,\ldots,y_{5})\in \Z^5$, we define
\begin{align}\label{defineg}g(\mathbf{x},\mathbf{y})=\sum_{i=1}^5x_ih_i(\mathbf{y}),\end{align}
where for $1\le i\le 5$, $h_i(\mathbf{y})$ denotes
\begin{align}\label{definehi}h_i(\mathbf{y})=2\sum_{j=1}^{5}a_{i,j+5}y_j.\end{align}

For $\mathbf{y}=(y_1,\ldots,y_{5})\in \Z^5$, denote by $\lambda(\mathbf{y})=\prod_{i=1}^5\lambda_{i+5}(y_{i})$.
Recalling \eqref{defSalpha}, we deduce by triangular inequality that
\begin{align}\label{boundStoS0}S(\alpha)\le
\sum_{x_{11},\ldots,x_{s}}\Big(\prod_{i=11}^{s}|\lambda_i(x_i)|\Big)|S_0(\alpha),\end{align}
where $S_0(\alpha):=S_0(\alpha;x_{11},\ldots,x_{s})$ is
\begin{align*}S_0(\alpha)=\sum_{\mathbf{x}}\Lambda_{b,W;I}(\mathbf{x})\Big|\sum_{\mathbf{y}}\lambda(\mathbf{y})e\big(\alpha f(\mathbf{x},\mathbf{y},x_{11},\ldots,x_s)\big)\Big|.\end{align*}
Note that if $s=10$, then we just have $S(\alpha)\le S_0(\alpha)$. And in the case $s\ge 11$, $S_0(\alpha)$ may depend on $x_{11},\ldots,x_s$.

Now we install a smooth weight $w(t)$ supported on $[Y_0-Y, Y+2Y]$ satisfying

$\ \ \ \ $(i) $w(t)\ge 0$ for all $t$,

$\ \ \ \ $(ii) $w(t)\ge 1$ for $Y_0\le t\le Y_0+Y$, and

$\ \ \ \ $(iii) $w^{(2)}(t)\ll Y^{-2}$ for all $t>0$,

\noindent  where $w^{(2)}(t)$ means the second derivative. For example, we may choose $w_0(x)$ to be
\begin{align*}w_0(x)=\begin{cases}\exp(1/2)\exp(\frac{1}{(x-5/2)^2-9/4}) \ \ \ &\textrm{if } \ 1<x<4,
\\ 0 \ \ \ &\textrm{otherwise},\end{cases}\end{align*}
and define
\begin{align*}w(t)=w_0\big((t-Y_0+2Y)/Y\big).\end{align*}
Then we introduce
\begin{align}\label{defineLambdawbWI}\Lambda_{w;b,W;I}(x)=w(Wx+b)\Lambda_{b,W;I}(x).\end{align}

We deduce by Cauchy's inequality and the definition of $w(t)$ that
\begin{align*}|S_0(\alpha)|^2\ll &
\frac{Y^5}{\phi(W)^{5}}\sum_{\mathbf{x}}\Lambda_{b,W;I}(\mathbf{x})\Big|\sum_{\mathbf{y}}\lambda(\mathbf{y})
e\big(\alpha f(\mathbf{x},\mathbf{y},x_{11},\ldots,x_n)\big)\Big|^2
\\ \ll & \frac{Y^5}{\phi(W)^{5}}\sum_{\mathbf{x}}\Lambda_{w;b,W;I}(\mathbf{x})\Big|\sum_{\mathbf{y}}\lambda(\mathbf{y})
e\big(\alpha f(\mathbf{x},\mathbf{y},x_{11},\ldots,x_n)\big)\Big|^2
.\end{align*}
On expanding the square and exchanging the order of summations, we deduce from above
\begin{align}\label{boundS0afterchangesum}|S_0(\alpha)|^2\ll
\frac{Y^5}{\phi(W)^{5}}\sum_{\mathbf{y}}\sum_{\mathbf{z}}\lambda(\mathbf{y})\lambda(\mathbf{z})
\sum_{\mathbf{x}}\Lambda_{w;b,W;I}(\mathbf{x})e\big(\alpha G(\mathbf{x},\mathbf{y},\mathbf{z})\big),\end{align}
where $G(\mathbf{x},\mathbf{y},\mathbf{z})=G(\mathbf{x},\mathbf{y},\mathbf{z},x_{11},\ldots,x_{n})$ is
\begin{align*}G(\mathbf{x},\mathbf{y},\mathbf{z})
=f(\mathbf{x},\mathbf{y},x_{11},\ldots,x_n)-f(\mathbf{x},\mathbf{z},x_{11},\ldots,x_n).\end{align*}
On recalling \eqref{defineg} and \eqref{definehi}, we observe
\begin{align*}\Big|\sum_{\mathbf{x}}\Lambda_{w;b,W;I}(\mathbf{x})e\big(\alpha G(\mathbf{x},\mathbf{y},\mathbf{z})\big)\Big|=&
\Big|\sum_{\mathbf{x}}\Lambda_{w;b,W;I}(\mathbf{x})e\big(\alpha g(\mathbf{x},\mathbf{y}-\mathbf{z})\big)\Big|,
\end{align*}
and therefore,
\begin{align}\Big|\sum_{\mathbf{x}}\Lambda_{w;b,W;I}(\mathbf{x})e\big(\alpha G(\mathbf{x},\mathbf{y},\mathbf{z})\big)\Big|=\prod_{i=1}^5\Big|\sum_{x}\Lambda_{w;b,W;I}(x)e\big(\alpha xh_i(\mathbf{y}-\mathbf{z})\big)\Big|.\label{independentofx11}
\end{align}

Now we conclude from \eqref{boundS0afterchangesum} and \eqref{independentofx11} that
\begin{align}\label{boundS0toR}|S_0(\alpha)|^2\ll \big(Y/\phi(W)\big)^5R(\alpha),\end{align}
where
\begin{align}\label{defineRalpha}R(\alpha)=\sum_{\mathbf{y}}\sum_{\mathbf{z}}\Lambda_{b,W;I}(\mathbf{y})\Lambda_{b,W;I}(\mathbf{z})
\prod_{i=1}^5\Big|\sum_{x}\Lambda_{w;b,W;I}(x)e\big(\alpha xh_i(\mathbf{y}-\mathbf{z})\big)\Big|.\end{align}
On recalling \eqref{assumptiononlambdai}, we have
$$\sum_{x_{11},\ldots,x_{s}}\Big(\prod_{i=11}^{s}|\lambda_i(x_i)|\Big)\ll \big(\delta Y/\phi(W)\big)^{s-10}.$$
Since $R(\alpha)$ is independent of $x_{11},\ldots,x_s$ (if $s\ge 11$), by \eqref{boundStoS0} and \eqref{boundS0toR}, we arrive at the following result.
\begin{lemma}\label{lemmaStoR}Let $S(\alpha)$ be defined in \eqref{defSalpha}. Then we have
\begin{align*}S(\alpha)\ll \frac{\delta^{s-10}Y^{s-10}}{\phi(W)^{s-10}}\cdot \frac{Y^\frac{5}{2}}{\phi(W)^\frac{5}{2}}R(\alpha)^{\frac{1}{2}},\end{align*}
where $R(\alpha)$ is given in \eqref{defineRalpha}.
\end{lemma}

Throughout Sections 5-6, we shall assume that $\alpha$ has the rational approximation
\begin{align}\label{assumptiononalpha}\alpha=\frac{a}{q}+\beta,\ 1\le a\le q\le Q_0, \ (a,q)=1\ \textrm{ and } |\beta|\le \frac{Q_0W^2}{Y^2}.\end{align}

Now we consider the innermost summation in \eqref{defineRalpha}.
\begin{lemma}\label{lemmaTWpartial}Let $\alpha=a/q+\beta$ with $a,q$ and $\beta$ satisfying \eqref{assumptiononalpha}. Let $h\in \Z$.
Then we have
\begin{align}\sum_{x}\Lambda_{w;b,W;I}(x)e\big(\alpha xh\big)
= &e\big(-\frac{\alpha bh}{W}\big)
\frac{R^\ast_{W,b}(q,a,h)}{\phi(qW)}\int w(x)e\big(\frac{\beta xh}{W}\big)dx\notag
\\  & \ +O\big(\frac{Y}{(\log X)^{400s}}\big),\label{asympsumx}\end{align}
where $R^\ast_{W,b}(q,a,h)$ is given in \eqref{defineTWq}.
\end{lemma}
\begin{proof}
Note that
\begin{align*}\sum_{x}\Lambda_{w;b,W;I}(x)e\big(\alpha xh\big)=
\sum_{\substack{x\equiv b\pmod{W}}}w(x)\Lambda(x)e\big(\frac{\alpha xh-\alpha bh}{W}\big).\end{align*}
We deduce by the Siegel-Walfisz theorem and the partial summation formula
\begin{align*}
\sum_{\substack{x\equiv b\pmod{W}}}w(x)\Lambda(x)e\big(\frac{\alpha xh}{W}\big)=&\frac{1}{\phi(qW)}R^\ast_{W,b}(q,a,h)\int w(x)e\big(\frac{\beta xh}{W}\big)dx
\\ & \ \ +O(Y(\log X)^{-400s}).\end{align*}
This completes the proof.
\end{proof}

The exponential sum $R^\ast_{W,b}(q,a,h)$ has been studied in Lemma \ref{lemmaboundTWq}. We next consider the integration in \eqref{asympsumx}.
\begin{lemma}\label{lemmaintbound}We have
\begin{align*}\int w(x)e\big(\frac{\beta xh}{W}\big)dx
\ll \min( Y,\ |\frac{\beta h}{W}|^{-2}Y^{-1}).\end{align*}
\end{lemma}
\begin{proof}We first observe a trivial bound
$$\int w(x)e\big(\frac{\beta xh}{W}\big)dx
\ll Y.$$
 For $\beta h\not=0$, we deduce by integration by parts twice that
\begin{align*}\int w(x)e\big(\frac{\beta xh}{W}\big)dx
=& -\frac{1}{2\pi i\frac{\beta h}{W}}\int w'(x)e\big(\frac{\beta xh}{W}\big)dx
\\ =&\,\frac{1}{(2\pi i\frac{\beta h}{W})^2}\int w^{(2)}(x)e\big(\frac{\beta xh}{W}\big)dx
.\end{align*}
On recalling $w^{(2)}(x) \ll Y^{-2}$, we obtain
\begin{align*}\int w(x)e\big(\frac{\beta xh}{W}\big)dx
\ll |\frac{\beta h}{W}|^{-2}Y^{-1}
.\end{align*}
We complete the proof.
\end{proof}

In view of Lemma \ref{lemmaboundTWq} and Lemma \ref{lemmaintbound}, we introduce
\begin{align}\label{defineeta}\eta(u)=(u,q)\min(Y,\ |\frac{\beta u}{W}|^{-2}Y^{-1}),\end{align}
and by Lemma \ref{lemmaTWpartial}, we have
\begin{align}\label{boundsumx}\sum_{x}\Lambda_{w;b,W;I}(x)e\big(\alpha xh\big)
\ll \frac{1}{\phi(qW)}\eta(h)
+O\big(\frac{Y}{(\log X)^{400s}}\big).\end{align}
We  define
\begin{align}\label{definePsi}\Psi(\mathbf{u})=\sum_{\substack{
\mathbf{y},\mathbf{z} \\ \eqref{conditionnonzero}
\\ h_j(\mathbf{y}-\mathbf{z})=u_j(1\le j\le 5) }}\Lambda_{b,W;I}(\mathbf{y})\Lambda_{b,W;I}( \mathbf{z})
,\end{align}
where the condition \eqref{conditionnonzero} in the above summation means \begin{align}\label{conditionnonzero} y_j-z_j\not=0 \ \textrm{ for all }\ 1\le j\le 5.\end{align}
\begin{lemma}\label{lemmaboundRfirst}Let $R(\alpha)$ be given in \eqref{defineRalpha}. Then
we have\begin{align}\label{boundRfirst}R(\alpha)\ll \frac{1}{\phi(qW)^5}
\sum_{X^{1/3}< |\mathbf{u}|\ll \frac{Y}{W}}\eta(\mathbf{u})\Psi(\mathbf{u})+\frac{Y^{15}}{(\log X)^{400s}}
,\end{align}
where $\eta(\mathbf{u})$ and $\Psi(\mathbf{u})$ are defined in \eqref{defineeta} and \eqref{definePsi}, respectively.
\end{lemma}\begin{proof} The contribution to the summations in \eqref{defineRalpha} from those terms with $y_j-z_j=0$ for some $1\le j\le 5$ is at most $O(X^{14}\log X)$.
Thus, we shall assume that \eqref{conditionnonzero} holds.  Similarly, the contribution to the summations in \eqref{defineRalpha} with
 $|h_j(\mathbf{y}-\mathbf{z})|\le X^{1/3}$ for some $1\le j\le 5$ is at most $O(X^{14+1/3}(\log X)^{15})$. We further assume
\begin{align}\label{conditionsizeh}
 |h_j(\mathbf{y}-\mathbf{z})|>X^{1/3}\ \textrm{ for all }\ 1\le j\le 5.\end{align}
Then we conclude from \eqref{defineRalpha} and  \eqref{boundsumx} that
\begin{align*}R(\alpha)\ll \frac{1}{\phi(qW)^5}\sum_{\substack{\mathbf{y},\ \mathbf{z}\\ \eqref{conditionnonzero},\ \eqref{conditionsizeh}}}\Lambda_{b,W;I}(\mathbf{y})\Lambda_{b,W;I}( \mathbf{z})
\prod_{j=1}^{5}\xi_j(\mathbf{y}-\mathbf{z})+\frac{Y^{15}}{(\log X)^{400s}}
,\end{align*}where $\xi_j(\mathbf{v})$ is defined as
\begin{align*}\xi_j(\mathbf{v})=(h_j(\mathbf{v}),q)\min(Y,\ |\frac{\beta h_j(\mathbf{v})}{W}|^{-2}Y^{-1}).\end{align*}
This completes the proof of \eqref{boundRfirst} by changing variables $u_i=h_i(\mathbf{y}-\mathbf{z})$ ($1\le i\le 5$).
\end{proof}

\vskip3mm

\section{Restriction estimate: an application of the sieve}

In order to deal with $\Psi(\mathbf{u})$ defined in \eqref{definePsi}, we need an upper bound for
\begin{align}\label{defineUps}\Upsilon(b,W;I;v),\end{align}
 which denotes the number of solutions to $p_1-p_2=Wv$, where $p_1,p_2\in I$ are primes satisfying $p_1\equiv p_2\equiv b\pmod{W}$. We can obtain a nice upper bound via a standard application of the sieve method. The result in the case $W=1$ is well-known, and its proof works well to deal with the general case if $W$ is no more than a fixed power of $\log X$. So we give the proof of the following result briefly.

\begin{lemma} \label{lemmasieve}Let $v\not=0$.  Let $b,W$ and $I$ be given in \eqref{assumptionbW} and \eqref{defineintervalI}, respectively. Let $\Upsilon(b,W;I;v)$ be given around \eqref{defineUps}. Then we have
\begin{align}\label{upperboundtwin}\Upsilon(b,W;I;v) \ll \frac{Y}{\phi(W)(\log X)^2}\rho(Wv)
,\end{align}where for a nonzero integer $x$, $\rho(x)$ denotes
\begin{align*}\rho(x)=\prod_{p|x}(1+\frac{1}{p})
.\end{align*}
\end{lemma}
\begin{proof} In order to apply the sieve method, for $(d,2Wv)=1$, we consider
\begin{align*}A_d=\big\{p+Wv:\ p\in I,\ p\equiv b\pmod{W}\ \textrm{ and }\ p+Wv\equiv 0\pmod{d}
\big\}.\end{align*}
We may expect that $|A_d|$ is well approximated by $\frac{1}{\phi(d)}X_0$ (at least in some average sense), where
\begin{align*}X_0=\frac{1}{\phi(W)}\int_{Y_0}^{Y_0+Y}\frac{1}{\log t}dt,\end{align*}
and thus we consider
\begin{align*}r_d=|A_d|-\frac{1}{\phi(d)}X_0.\end{align*}
We deduce by  Bombieri-Vinogradov theorem and Cauchy's inequality that for any constant $C>100$,
\begin{align*}\sum_{\substack{d\le X^{\frac{1}{2}-\varepsilon}\\ (d,2Wv)=1}}\tau(d)^2|r_d| \ll \frac{X}{(\log X)^{C}}.\end{align*}
On choosing $z=X^{1/5}$, we deduce from Theorem 7.1 in \cite{DHG} that
\begin{align}\label{upperboundbysieve}|\{n\in A_1:\ (n,P_z)=1\}| \ll X_0 \prod_{\substack{3\le p\le z \\ p\nmid Wv}}(1-\frac{1}{p-1}),\end{align}
where $P_z$ denotes
\begin{align*}P_z=\prod_{\substack{3\le p\le z \\ p\nmid Wv}}p.\end{align*}
Now \eqref{upperboundtwin} follows from \eqref{upperboundbysieve} since $\Upsilon(b,W;I;v) \le |\{n\in A_1:\ (n,P_z)=1\}|$.
\end{proof}

A simple upper bound for $\rho(v)$ asserts that $\rho(v)\ll \log\log X$. However, as explained in the introduction, our proof of Theorem \ref{theorem1} would fail if there were an extra factor $\log\log X$ in Proposition \ref{proprestriction}. Therefore, we need to prepare a technical lemma to show that $\rho(v)$ exhibits like a constant on average. We first point out for $0<|v|<X$,
\begin{align}\label{lemmaboundrhox}\rho(v)\ll \sum_{\substack{d\le \log X
\\ d|v}} \frac{1}{d}
.\end{align}
For $\mathbf{v}\in \Z^5$, we introduce the conditions
\begin{align}\label{condition1}
0<|v_j|<Y,\, H_j<|h_j(\mathbf{v})|\le2H_j\, \textrm{ and } \, h_j(\mathbf{v})\equiv c_j\pmod{q}\, \textrm{ for }\, j=1,\ldots,5.\end{align}
Then we introduce
\begin{align}\label{definecalH}
\mathcal{H}:=\mathcal{H}(H_1,\ldots,H_5;c_1,\ldots,c_5;q)=\sum_{\substack{\mathbf{v}
\\ \eqref{condition1}}}\rho(\mathbf{v}),\end{align}
where the summation is taken over $\mathbf{v}$ satisfying conditions in \eqref{condition1}.
\begin{lemma}\label{lemmacalH}Suppose that $H_i\ge q\log X$ for all $1\le i\le 5$. Let $\mathcal{H}$ be defined in \eqref{definecalH}. Then we have
\begin{align*}
\mathcal{H}\ll H_1H_2H_3H_4H_5 q^{-5}\tau(q).\end{align*}
\end{lemma}
\begin{proof}By \eqref{lemmaboundrhox}, we have
\begin{align*}\mathcal{H}\ll   \sum_{d_1,d_2,d_3,d_4,d_5\le \log X} \frac{1}{d_1d_2d_3d_4d_5}\sum_{\substack{v_1,\ldots,v_5
\\ \eqref{condition1}\\ d_j|v_j (1\le j\le 5)}}1
.\end{align*}
We shall change variables by $\mathbf{u}=2M_0\mathbf{v}$, where $M_0$ is given in \eqref{defineM0}. In particular,
$u_i=h_i(\mathbf{v})$ for $1\le i\le 5$. Note that $\det(2M_0)\mathbf{v}=(2M_0)^\ast \mathbf{u}$, and we write
\begin{align*} (2M_0)^\ast =(b_{i,j})_{1\le i,j\le 5}=\begin{pmatrix}b_{1,1}  & \cdots & b_{1,5}
\\ \vdots & \cdots & \vdots
\\ b_{5,1} & \cdots  & b_{5,5} \end{pmatrix}.\end{align*}
For each $1\le j\le 5$, the condition $d_j|v_j$ implies
$b_{j,1}u_1+\cdots +b_{j,5}u_5\equiv 0\pmod{d_j}$. Then we deduce that
\begin{align*}\mathcal{H}\ll \sum_{d_1,d_2,d_3,d_4,d_5\le \log X } \frac{1}{d_1d_2d_3d_4d_5}\sum_{\substack{u_1,\ldots,u_5
\\ \eqref{condition2}
\\ b_{j,1}u_1+\cdots +b_{j,5}u_5\equiv 0\pmod{d_j}(1\le j\le 5)}}1
,\end{align*}
where the condition \eqref{condition2} is
\begin{align}\label{condition2}
 H_j<|u_j|\le2H_j \ \textrm{ and } \ u_j\equiv c_j\pmod{q}\ \textrm{ for all }\ 1\le j\le 5.\end{align}

By symmetry, we only need to prove
\begin{align}\label{boundcalH0}
\mathcal{H}_0\ll H_1H_2H_3H_4H_5 q^{-5}\tau(q),\end{align}
where\begin{align*}\mathcal{H}_0=\sum_{d_5\le \log X}\sum_{d_1,d_2,d_3,d_4\le d_5} \frac{1}{d_1d_2d_3d_4d_5}\sum_{\substack{u_1,\ldots,u_5
\\ \eqref{condition2}
\\ b_{j,1}u_1+\cdots +b_{j,5}u_5\equiv 0\pmod{d_j}(1\le j\le 5)}}1
.\end{align*}
We omit congruences modulo $d_j$ ($1\le j\le 4$) to deduce that
\begin{align*}\mathcal{H}_0\le\sum_{d_5\le \log X}\sum_{d_1,d_2,d_3,d_4\le d_5} \frac{1}{d_1d_2d_3d_4d_5}\sum_{\substack{u_1,\ldots,u_5
\\ \eqref{condition2}
\\ b_{5,1}u_1+\cdots +b_{5,5}u_5\equiv 0\pmod{d_5}}}1
.\end{align*}
Now the above innermost summation is independent of $d_1,\ldots,d_4$, and the summations over $d_1,\ldots,d_4$ contribute at most $(1+\log d_5)^4$ to $\mathcal{H}_0$. Then we obtain
\begin{align}\label{boundH0-1}\mathcal{H}_0 \ll
 \sum_{d_5\le \log X} \frac{(1+\log d_5)^4}{d_5}\sum_{\substack{u_1,\ldots,u_5
\\ \eqref{condition2}
\\ b_{5,1}u_1+\cdots +b_{5,5}u_5\equiv 0\pmod{d_5}}}1
.\end{align}
Note that $\rank(M_0)=5$, and thus at least one of $b_{5,1},\ldots, b_{5,5}$ is nonzero. Without loss of generality, we assume $b_{5,1}\not=0$.
Then for any fixed $u_2,u_3,u_4,u_5$, there are at most $O(\frac{H_1}{[q,d_5]})$ possible choices of $u_1$ due to the congruences modulo $q$ and $d_5$, respectively.
Therefore, we obtain
\begin{align}\label{boundH0inner}\sum_{\substack{u_1,\ldots,u_5
\\ \eqref{condition2}
\\ b_{5,1}u_1+\cdots +b_{5,5}u_5\equiv 0\pmod{d_5}}}1
\ll \frac{H_1H_2H_3H_4H_5}{q^4[q,d_5]}.\end{align}
We conclude \eqref{boundcalH0} from \eqref{boundH0-1} and \eqref{boundH0inner} in combination with the following elementary inequality
$$\sum_{d\le \log X}\frac{\log d}{d[q,d]}\ll q^{-1}\tau(q).$$
This completes the proof.
\end{proof}

\begin{lemma}\label{lemmaRboundfinal}Let $R(\alpha)$ be given in \eqref{defineRalpha} with $\alpha$ satisfying \eqref{assumptiononalpha}.
Then we have\begin{align*}R(\alpha)\ll \frac{Y^{15}\tau(q)^6}{\phi(q)^5\phi(W)^{15}}
\big(1+(Y/W)^2|\beta|\big)^{-5}.
\end{align*}
\end{lemma}
\begin{proof}By Lemma \ref{lemmaboundRfirst}, we first introduce congruence conditions to conclude
\begin{align*}R(\alpha)\ll \frac{1}{\phi(qW)^5}
\sum_{1\le \mathbf{c}\le q}(\mathbf{c},q)
\sum_{\substack{X^{1/3}< |\mathbf{u}|\ll \frac{Y}{W}
 \\ \mathbf{u}\equiv \mathbf{c}\pmod{q}}}\theta(\mathbf{u})\Psi(\mathbf{u})+\frac{Y^{15}}{(\log X)^{400s}}
,\end{align*}
where
\begin{align}\label{definetheta}\theta(\mathbf{u})=\prod_{j=1}^5\min(Y,\ |\frac{\beta u_j}{W}|^{-2}Y^{-1}).\end{align}
By the dyadic argument, we have
\begin{align}\label{boundRalpha6}R(\alpha)\ll \frac{1}{\phi(qW)^5}
\sum_{1\le \mathbf{c}\le q}(\mathbf{c},q)
\sum_{\substack{ H_1,\ldots,H_5 \ll Y/W \\ H_i=2^{l_i}X^{1/3} (1\le i\le 5)}}\theta(H_1,\ldots,H_5)\mathfrak{H}+\frac{Y^{15}}{(\log X)^{400s}}
,\end{align}
where the innermost multiple summation is taken over $H_i$ in the form $H_i=2^{l_i}X^{1/3}$ with integers $l_i\ge 0$, and $\mathfrak{H}:=\mathfrak{H}(H_1,\ldots,H_5;c_1,\ldots,c_5;q)$ is
\begin{align}\label{definefrakH}\mathfrak{H}=
\sum_{\substack{H_i<|u_i|\le 2H_i (1\le i\le 5)
 \\ \mathbf{u}\equiv \mathbf{c}\pmod{q}}}\Psi(\mathbf{u})
.\end{align}
On recalling the definition \eqref{definePsi}, we have
\begin{align}\label{boundPsi}\Psi(\mathbf{u})=\sum_{\substack{
1\le |\mathbf{v}|\le Y/W
\\ h_i(\mathbf{v})=u_i(1\le i\le 5) }} \sum_{\substack{
\mathbf{y},\mathbf{z}
\\ \mathbf{y}-\mathbf{z}=\mathbf{v} }}\Lambda_{b,W;I}(\mathbf{y})\Lambda_{b,W;I}( \mathbf{z})
,\end{align}
where $1\le |\mathbf{v}|\le Y/W$ here means $1\le |v_j|\le Y/W$ for all $1\le j\le 5$.

We deduce by Lemma \ref{lemmasieve} that
\begin{align}\label{boundinnerPsi}\sum_{\substack{
\mathbf{y},\mathbf{z}
\\ \mathbf{y}-\mathbf{z}=\mathbf{v} }}\Lambda_{b,W;I}(\mathbf{y})\Lambda_{b,W;I}( \mathbf{z}) \ll
\frac{Y^5\rho(W)^5}{\phi(W)^5}\rho(\mathbf{v}).\end{align}
Now we conclude from \eqref{definefrakH}, \eqref{boundPsi} and \eqref{boundinnerPsi} that
\begin{align}\label{boundfrakH}\mathfrak{H}\ll \frac{Y^5\rho(W)^5}{\phi(W)^5}
\sum_{\substack{H_i<|u_i|\le 2H_i (1\le i\le 5)
 \\ \mathbf{u}\equiv \mathbf{c}\pmod{q}}}\sum_{\substack{
1\le |\mathbf{v}|\le Y/W
\\ h_i(\mathbf{v})=u_i(1\le i\le 5) }}\rho(\mathbf{v})
.\end{align}
Note that the multiple summations in \eqref{boundfrakH} coincide with the definition of $\mathcal{H}$ in \eqref{definecalH}. Therefore, applying Lemma \ref{lemmacalH}, we obtain
\begin{align}\label{boundfrakH2}
\mathfrak{H}\ll \frac{Y^5\rho(W)^5}{\phi(W)^5}H_1H_2H_3H_4H_5 q^{-5}\tau(q).\end{align}
We put \eqref{boundfrakH2} into \eqref{boundRalpha6} to deduce that
\begin{align}\label{boundRalpha}R(\alpha)\ll \frac{Y^5\rho(W)^5\tau(q)^6}{\phi(qW)^5\phi(W)^5}\cdot \mathcal{T}^5+\frac{Y^{15}}{(\log X)^{400s}},\end{align}
where
\begin{align*}\mathcal{T}=
\sum_{\substack{ X^{1/3} <H\ll Y/W \\ H=2^{l}X^{1/3} }}H\min(Y,\ |\frac{\beta H}{W}|^{-2}Y^{-1}) .\end{align*}

To deal with $\mathcal{T}$, we first consider the case $|\beta|\le (Y/W)^{-2}$, and deduce that
\begin{align}\label{boundT1}\mathcal{T} \ll
\sum_{\substack{ X^{1/3} \le H\ll Y/W \\ H=2^{l}X^{1/3} }}HY \ll \frac{Y^2}{W}.\end{align}
If $|\beta|> (Y/W)^{-2}$, then we have
\begin{align}\label{boundT2}\mathcal{T}
 \ll
\sum_{\substack{ X^{1/3} < H \le \frac{W}{|\beta|Y}\\ H=2^{l}X^{1/3} }}HY +
\sum_{\substack{ H> \frac{W}{|\beta|Y}\\ H=2^{l}X^{1/3} }}\frac{W^2}{H|\beta|^2Y}\ll \frac{W}{|\beta|}.\end{align}
Therefore, we conclude from \eqref{boundT1} and \eqref{boundT2} that
\begin{align}\label{boundcalT}\mathcal{T} \ll \frac{Y^2}{W}\big(1+(Y/W)^2|\beta|\big)^{-1}.\end{align}
Finally, we obtain by \eqref{boundRalpha} and \eqref{boundcalT} that
\begin{align*}R(\alpha)\ll \frac{Y^{15}\rho(W)^5\tau(q)^6}{\phi(q)^5\phi(W)^{10}W^5}
\big(1+(Y/W)^2|\beta|\big)^{-5}+\frac{Y^{15}}{(\log X)^{400s}}.\end{align*}
In view of \eqref{assumptiononalpha}, the above estimate holds with $Y^{15}(\log X)^{-400s}$ omitted.
This completes the proof on noting that $\rho(W)/W\ll \phi(W)$.
\end{proof}

\vskip3mm

\noindent {\it Proof of Proposition \ref{proprestriction}}. In view of Lemma \ref{lemmaliu2}, we only need to prove
\begin{align*}\int_{\mathfrak{M}(Q_0)\setminus \mathfrak{M}(Q)}|S(\alpha)|d\alpha\ll \frac{\delta^{s-10}W^2Y^{s-2}}{\phi(W)^s}Q^{-\frac{10}{21}}.\end{align*}
For $\alpha\in \mathfrak{M}(Q_0)$, we can represent $\alpha$ uniquely in the form $\alpha=\frac{a}{q}+\beta$ with $a,q$ and $\beta$ satisfying \eqref{assumptiononalpha}. By Lemma \ref{lemmaStoR} and Lemma \ref{lemmaRboundfinal}, we obtain
\begin{align*}S(\alpha)\ll \frac{\delta^{s-10} Y^{s}}{\phi(W)^{s}}q^{-\frac{5}{2}+\varepsilon}
\big(1+(Y/W)^2|\beta|\big)^{-\frac{5}{2}}.\end{align*}

For $\alpha=\frac{a}{q}+\beta\in \mathfrak{M}(Q_0)\setminus \mathfrak{M}(Q)$, we have
\begin{align*}
\big(q+q(Y/W)^2|\beta|\big)^{-\frac{10}{21}}\ll Q^{-\frac{10}{21}},\end{align*}
and we deduce that
\begin{align*}\int_{\mathfrak{M}(Q_0)\setminus \mathfrak{M}(Q)}|S(\alpha)|d\alpha\ll & \frac{\delta^{s-10} Y^{s}}{\phi(W)^{s}} Q^{-\frac{10}{21}}
\int_{\mathfrak{M}(Q_0)}\big(q+q(Y/W)^2|\beta|\big)^{-2-\frac{1}{42}+\varepsilon}d\alpha
\\  \ll & \frac{\delta^{s-10}W^2Y^{s-2}}{\phi(W)^s}Q^{-\frac{10}{21}}.\end{align*}
This completes the proof of Proposition \ref{proprestriction}.

\vskip3mm

\section{Roth's density increment argument}

In this section, we always assume $f$ is translation invariant.
For $\mathcal{A}\subseteq \mathcal{P}\cap [X/2,\,X]$, we define \begin{align*}
\nu_f^\ast (\mathcal{A})=\sum_{\substack{x_1,\ldots,x_s\in \mathcal{A}\cap I\\ f(\mathbf{x})=0 \\ \mathbf{x}\equiv b\pmod{W} }}\Lambda(x_1)\cdots \Lambda(x_s).\end{align*}
Recalling the definition of $\mathcal{A}'$ in \eqref{defineAp}, we can represent $\nu_f^\ast(\mathcal{A})$ in the form
\begin{align*}
\nu_f^\ast(\mathcal{A})=
\sum_{\substack{\mathbf{x}\\ f(\mathbf{x})=0  }}\prod_{j=1}^s\Big(1_{\mathcal{A}'}\cdot\Lambda_{b,W;I}(x_j)\Big).
\end{align*}

For $(c,r)=1$ and $\mathcal{B}\subseteq \N$, we define
\begin{align*}\mathcal{N}_{c,r;I}(\mathcal{B})=
\sum_{\substack{x\equiv c\pmod{r}\\ x\in \mathcal{B}\cap I}}\Lambda(x).\end{align*}
Now we introduce the (relative) density
\begin{align}\label{definerelativedensity}\delta_{c,r;I}(\mathcal{B})
=\frac{\mathcal{N}_{c,r;I}(\mathcal{B})}{\mathcal{N}_{c,r;I}(\N)}.\end{align}

On choosing
$$\delta=\delta_{b,W;I}(\mathcal{A}),$$
we obtain
\begin{align}\label{choosedelta}
\sum_{\substack{x }}1_{\mathcal{A}'}\cdot\Lambda_{b,W;I}(x)-\delta\sum_{x}\Lambda_{b,W;I}(x)
=0.
\end{align}
Therefore, we may compare the function $1_{\mathcal{A}'}\cdot\Lambda_{b,W;I}(x)$ with $\delta\Lambda_{b,W;I}(x)$, and for this purpose, we introduce
\begin{align}\label{definevarpi}\varpi=1_{\mathcal{A}'}\cdot \Lambda_{b,W;I}-\delta\Lambda_{b,W;I} .\end{align}
Note that \begin{align*}
\nu_f(b,W;I)=
\sum_{\substack{\mathbf{x}\\ f(\mathbf{x})=0  }}\prod_{j=1}^s\Lambda_{b,W;I}(x_j).
\end{align*}
Then on replacing $1_{\mathcal{A}'}\cdot \Lambda_{b,W;I}$ by $\varpi+\delta\Lambda_{b,W;I}$, we can represent $\nu_f^\ast(\mathcal{A})-\delta^s\nu_f(b,W;I)$ in the form
\begin{align}\label{representnu}
\nu_f^\ast(\mathcal{A})-\delta^s\nu_f(b,W;I)=\sum_{\varpi_1,\ldots,\varpi_s}\nu_f^\ast(\varpi_1,\ldots,\varpi_s),
\end{align}
where\begin{align*}\nu_f^\ast(\varpi_1,\ldots,\varpi_s)=
\sum_{\substack{\mathbf{x}\\ f(\mathbf{x})=0  }}\prod_{j=1}^s\varpi_j(x_j)
\end{align*}
and the summation in \eqref{representnu} is taken over
\begin{align}\label{conditionvarpi}\varpi_1,\ldots,\varpi_s\in \{\varpi,\ \delta\Lambda_{b,W;I}\} \ \textrm{ with }\ \ \varpi_j=\varpi \ \textrm{
for some }\ \ 1\le j\le s.\end{align}
For $\boldsymbol{\varpi}=(\varpi_1,\ldots,\varpi_s)$ satisfying \eqref{conditionvarpi}, we have
\begin{align}\label{circle2}\nu_f^\ast(\boldsymbol{\varpi})=\int_0^1S(\alpha;\boldsymbol{\varpi})d\alpha.
\end{align}

\begin{lemma}\label{lemmadensityincmajor}Suppose that $\boldsymbol{\varpi}=(\varpi_1,\ldots,\varpi_s)$ satisfies \eqref{conditionvarpi}. Let $\delta=\delta_{b,W;I}(\mathcal{A})$, and let $\delta\ge (\log X)^{-50}$.  Let $Q\le (\log X)^{50}$.
Suppose that
\begin{align}\label{lowermajor}\Big|\int_{\mathfrak{M}(Q)}S(\alpha;\boldsymbol{\varpi})d\alpha\Big|\gg \frac{\delta^sW^2Y^{s-2}}{\phi(W)^{s}}.\end{align}
Then there exist $c\in \Z$, $q\in \N$, a subinterval $J\subseteq I$ and $C_f>0$ such that
 \begin{align*}(c,Wq)=1,\ q\le Q, \ \ |J|\ge C_f\delta Q^{-3.1}|I|\ \ \textrm{ and }\ \ \delta_{c,Wq;J}(\mathcal{A})\ge \delta(1+C_fQ^{-3.1}).
 \end{align*}
\end{lemma}
\begin{proof}
Without loss of generality, we assume $\varpi_1=\varpi$. For $\alpha\in \mathfrak{M}(Q)$, one has the unique rational approximation
\begin{align}\label{repalpha}\alpha =\frac{a}{q}+\beta,\ \ 1\le a\le q\le Q, \ (a,q)=1\ \textrm{ and } |\beta|\le \frac{QW^2}{qY^2}.\end{align}
We have
\begin{align}\label{Salphavarpi1}
S(\frac{a}{q}+\beta;\boldsymbol{\varpi})=\sum_{x_2,\ldots,x_s}\prod_{j=2}^s\varpi_j(x_j)\Xi(x_2,\ldots,x_s;q,a,\beta),\end{align}
where
\begin{align*}\Xi(x_2,\ldots,x_s;q,a,\beta)=\sum_{\frac{Y_0-b}{W}<x_1\le \frac{Y_0-b+Y}{W}}
\varpi(x_1)e\big((\frac{a}{q}+\beta)f(x_1,x_2,\ldots,x_s)\big).\end{align*}
We introduce
\begin{align*}U(\gamma):=U(x_2,\ldots,x_s;q,a,\gamma)=
\sum_{\frac{Y_0-b}{W}<x_1\le \gamma}\varpi(x_1)e\big(\frac{a}{q}f(x_1,x_2,\ldots,x_s)\big),\end{align*}
and then we deduce by the partial summation formula that
\begin{align}\label{boundxi}\Xi(x_2,\ldots,x_s;q,a,\beta)=\int_{\frac{Y_0-b}{W}}^{\frac{Y_0-b+Y}{W}}e\big(\beta f(\gamma,x_2,\ldots,x_s)\big)d U(\gamma).\end{align}
Let
\begin{align*}U_{\max}=\sup_{x_2,\ldots,x_s}\,\sup_{q\le Q}\, \sup_{\substack{1\le a\le q
\\ (a,q)=1}}\,\sup_{\frac{Y_0-b}{W}<\gamma \le \frac{Y_0-b+Y}{W}}|U(x_2,\ldots,x_s;q,a,\gamma)|.\end{align*}
Since $f$ is translation invariant, for $\frac{Y_0-b}{W}\le \gamma \le \frac{Y_0-b+Y}{W}$, we have
\begin{align}\label{boundpartial}\frac{\partial}{\partial \gamma}e\big(\beta f(\gamma,x_2,\ldots,x_s)\big)\ll \frac{QW}{qY}.\end{align}
We obtain from \eqref{boundxi} and \eqref{boundpartial}
\begin{align}\label{boundXi}\Xi(x_2,\ldots,x_s;q,a,\beta)\ll \frac{Q}{q}U_{\max}.\end{align}
We make use of the estimate
\begin{align*}\sum_{\frac{Y_0-b}{W}<x\le \frac{Y_0-b+Y}{W}}
\varpi_j(x)\ll  \frac{\delta Y}{\phi(W)}\end{align*}
for $2\le j\le s$, and deduce from \eqref{Salphavarpi1} and \eqref{boundXi} that
\begin{align}\label{Salphavarpi2}
S(\frac{a}{q}+\beta;\boldsymbol{\varpi})\ll  \frac{\delta^{s-1}Y^{s-1}Q}{\phi(W)^{s-1}q}U_{\max}.\end{align}

For $1\le u\le q$, we introduce
\begin{align}\label{defineVqgammau}V(q,\gamma,u)=
\sum_{\substack{\frac{Y_0-b}{W}<x\le \gamma\\ x\equiv u\pmod{q}}}\varpi(x),\end{align}
and define
\begin{align*}V_{\max}=\sup_{q\le Q}\sup_{\frac{Y_0-b}{W}<\gamma \le \frac{Y_0-b+Y}{W}}\sup_{1\le u\le q}|V(q,\gamma,u)|.\end{align*}
We deduce that
\begin{align*}U(x_2,\ldots,x_s;q,a,\gamma)=\sum_{1\le u\le q}e\big(\frac{a}{q}f(u,x_2,\ldots,x_s)\big)V(q,\gamma,u)\ll qV_{\max}.\end{align*}
Therefore, by \eqref{Salphavarpi2}, we have
\begin{align}\label{Salphavarpi3}
S(\frac{a}{q}+\beta;\boldsymbol{\varpi})\ll \frac{\delta^{s-1}Y^{s-1}Q}{\phi(W)^{s-1}}V_{\max}.\end{align}

Note that
\begin{align*}\int_{\mathfrak{M}(Q)}S(\alpha;\boldsymbol{\varpi})d\alpha=\sum_{q\le Q}\sum_{\substack{a=1\\ (a,q)=1}}^q\int_{|\beta|\le \frac{Q}{q(Y/W)^2}}
S(\frac{a}{q}+\beta;\boldsymbol{\varpi})d\beta,\end{align*}
and by \eqref{Salphavarpi3} we have
\begin{align*}\int_{\mathfrak{M}(Q)}S(\alpha;\boldsymbol{\varpi})d\alpha\ll \frac{\delta^{s-1}Y^{s-3}W^2Q^3}{\phi(W)^{s-1}}V_{\max} .\end{align*}
In view of \eqref{lowermajor}, we obtain
\begin{align*}V_{\max}\gg \frac{\delta Y}{\phi(W)Q^{3}},\end{align*}
and therefore, there exist $q\le Q$, $\frac{Y_0-b}{W}<\gamma \le \frac{Y_0-b+Y}{W}$, $1\le u_0\le q$ and $c_f>0$ such that
\begin{align}\label{existVlower}|V(q,\gamma,u_0)|\ge c_f \frac{\delta Y}{\phi(W)Q^{3}}.\end{align}

For $1\le u\le q$, we define
\begin{align}\label{defineVqgammau2}V'(q,\gamma,u)=
\sum_{\substack{\gamma<x\le \frac{Y_0-b+Y}{W}\\ x\equiv u\pmod{q}}}\varpi(x).\end{align}
We claim that either
\begin{align}\label{claim1}V(q,\gamma,u)\ge c_f \frac{\delta Y}{\phi(W)Q^{3}}\cdot \frac{1}{2q}\ \textrm{ for some }\ 1\le u\le q,\end{align}
or
\begin{align}\label{claim2}V'(q,\gamma,u)\ge c_f \frac{Y}{\phi(W)Q^{3}}\cdot \frac{1}{2q}\ \textrm{ for some }\ 1\le u\le q.\end{align}
Otherwise, we have
\begin{align*}V(q,\gamma,u)< c_f \frac{Y}{\phi(W)Q^{3}}\cdot \frac{1}{2q}\ \textrm{ and }\  V'(q,\gamma,u)<c_f \frac{Y}{\phi(W)Q^{3}}\cdot \frac{1}{2q}\end{align*}
for all $1\le u\le q$.
 Then by \eqref{existVlower}, one has \begin{align*}V(q,\gamma,u_0)\le -c_f \frac{\delta Y}{\phi(W)Q^{3}}.\end{align*}
Therefore, there are at most $2q-1$ positive ones among $V(q,\gamma,u)$ and $V'(q,\gamma,u)$, and we deduce that
 \begin{align*}\sum_{\substack{u \\ V(q,\gamma,u)\ge 0}}V(q,\gamma,u)+\sum_{\substack{u \\ V'(q,\gamma,u)\ge 0}}V'(q,\gamma,u)< ( 2q-1)c_f \frac{\delta Y}{\phi(W)Q^{3}}\cdot \frac{1}{2q},\end{align*}
 and furthermore,
\begin{align}\label{sumallless}\sum_{1\le u\le q}\Big(V(q,\gamma,u)+V'(q,\gamma,u)\Big)<&( 2q-1)c_f \frac{\delta Y}{\phi(W)Q^{3}}\cdot \frac{1}{2q}-
c_f \frac{\delta Y}{\phi(W)Q^{3}}\notag
\\= & -c_f \frac{\delta Y}{\phi(W)Q^{3}}\cdot \frac{1}{2q}.\end{align}
By \eqref{choosedelta}, \eqref{defineVqgammau} and \eqref{defineVqgammau2}, we have
\begin{align}\label{sumall}\sum_{u}\Big(V(q,\gamma,u)+V'(q,\gamma,u)\Big)=\sum_{x}\varpi(x)=0,\end{align}
which is a contradiction to \eqref{sumallless}. Therefore, we have either \eqref{claim1} or \eqref{claim2} (or both).

Now by \eqref{claim1} or \eqref{claim2}, we can find $1\le u\le  q\le Q$ and a subinterval $J\subseteq I$ such that
\begin{align*}
\sum_{\substack{Wx+b\in J\\ x\equiv u\pmod{q}}}\varpi(x)\ge c_f\frac{\delta Y}{\phi(W)Q^3q}.\end{align*}
Note that
\begin{align*}
\sum_{\substack{Wy+b\in J\\ y\equiv u\pmod{q}}}\varpi(y)=\sum_{\substack{x\equiv b+uW\pmod{Wq}\\ x\in \mathcal{A}\cap J}}\Lambda(x)-
\delta\sum_{\substack{x\equiv b+uW\pmod{Wq}\\ x\in  J}}\Lambda(x),\end{align*}
and on writing $c=b+uW$, we have
\begin{align}\label{obtainincease}
\sum_{\substack{x\equiv c\pmod{Wq}\\ x\in \mathcal{A}\cap J}}\Lambda(x)-
\delta\sum_{\substack{x\equiv c\pmod{Wq}\\ x\in  J}}\Lambda(x)\ge c_f\frac{\delta Y}{\phi(W)Q^3q}.\end{align}
By \eqref{obtainincease}, we have
\begin{align}\label{lowerJ}\sum_{\substack{x\equiv c\pmod{Wq}\\ x\in  J}}\Lambda(x)\ll \frac{|J|}{\phi(Wq)}.\end{align}
Then we conclude from \eqref{obtainincease} and \eqref{lowerJ} that
\begin{align*}(c,Wq)=1 \ \textrm{ and } \ |J|\gg \delta Y Q^{-3.1}.\end{align*}
By \eqref{lowerJ}, we further have
\begin{align}\label{finddensityinc}\sum_{\substack{x\equiv c\pmod{Wq}\\ x\in  J}}\Lambda(x)\ll \delta^{-1}Q^{3.1}\frac{\delta Y}{\phi(W)Q^3q}.\end{align}
Now we deduce from \eqref{obtainincease} and \eqref{finddensityinc} that
\begin{align*}
\sum_{\substack{x\equiv c\pmod{Wq}\\ x\in \mathcal{A}\cap J}}\Lambda(x)-
\delta\sum_{\substack{x\equiv c\pmod{Wq}\\ x\in  J}}\Lambda(x)\ge C_f\delta Q^{-3.1}\sum_{\substack{x\equiv c\pmod{Wq}\\ x\in  J}}\Lambda(x)\end{align*}
for some $C_f>0$. This completes the proof.\end{proof}

Let $\upsilon_0(f)$ denote the number of solutions to \eqref{quadequation}, where $1\le x_1,\ldots,x_s\le X$ and $x_i=x_j$ for some $1\le i<j\le s$. Subject to the condition $\rank(M)\ge 5$, it is well-known that
\begin{align}\label{ups}\upsilon_0(f)\ll X^{s-3+\varepsilon}.\end{align}

\begin{lemma}\label{lemmadensityincrease}Suppose that there are no distinct primes $p_1,\ldots,p_s\in \mathcal{A}\cap I$ such that $f(p_1,\ldots,p_s)=0$. Let $\delta=\delta_{b,W;I}(\mathcal{A})$. Suppose that $\delta\ge (\log X)^{-1}$.
 Then there exist $c\in \Z$, $r\in \N$, a subinterval $J\subseteq I$ and a positive number $\alpha_0=\alpha_0(f)$ such that
 \begin{align*}(c,Wr)=1,\ r\le \alpha_0^{-1}\delta^{-21}, \ \ |J|\ge \alpha_0\delta^{70}|I|\ \ \textrm{ and }\ \ \delta_{c,Wr;J}(\mathcal{A})\ge \delta(1+\alpha_0\delta^{70}).
 \end{align*}
\end{lemma}
\begin{proof}By \eqref{ups}, we have
\begin{align*}\nu_f^\ast (\mathcal{A})\ll X^{s-3+\varepsilon}.\end{align*}
Corollary \ref{corollary13} yields
\begin{align*}\nu_f(b,W;I)\gg \frac{W^2Y^{s-2}}{\phi(W)^s}.\end{align*}
Then we have
\begin{align*}|\nu_f^\ast(\mathcal{A})-\delta^s\nu_f(b,W;I)|\gg \frac{\delta^sW^2Y^{s-2}}{\phi(W)^s}.\end{align*}
In view of \eqref{representnu}, one has
\begin{align}\label{nulower}|\nu_f^\ast(\boldsymbol{\varpi})|\gg \frac{\delta^s W^2 Y^{s-2}}{\phi(W)^s}\end{align}
for some $\boldsymbol{\varpi}$ satisfying \eqref{conditionvarpi}.
By \eqref{circle2}, \begin{align*}\nu_f^\ast(\boldsymbol{\varpi})
=\int_{\mathfrak{M}(Q)}S(\alpha;\boldsymbol{\varpi})d\alpha +\int_{\mathfrak{m}(Q)}S(\alpha;\boldsymbol{\varpi})d\alpha.\end{align*}
By Proposition \ref{proprestriction}, we have
\begin{align}\label{Salphavarpi4}\int_{\mathfrak{m}(Q)}S(\alpha;\boldsymbol{\varpi})d\alpha \ll \frac{\delta^{s-10} W^2Y^{s-2}}{\phi(W)^{s}}Q^{-10/21}.\end{align}
On choosing $Q=\delta^{-21}c_0(f)$ with some $c_0(f)>0$ sufficiently large, we obtain from \eqref{nulower} and \eqref{Salphavarpi4} that
\begin{align*}\Big|\int_{\mathfrak{M}(Q)}S(\alpha;\boldsymbol{\varpi})d\alpha\Big|\gg \frac{\delta^s W^2Y^{s-2}}{\phi(W)^s}.\end{align*}
This completes the proof on applying Lemma \ref{lemmadensityincmajor}.
\end{proof}

\noindent {\it Proof of Theorem \ref{theorem1}}.  Suppose that
\begin{align*}\frac{|\mathcal{A}|}{\pi(X)}\ge  3c_f(\log\log X)^{-\frac{1}{80}},\end{align*}
where $c_f$ is large in terms of $f$. By a standard dyadic argument, without loss of generality, we may assume $\mathcal{A}\subseteq (X/2,X]$ with
\begin{align}\label{density2}\frac{|\mathcal{A}|}{|\mathcal{P}\cap (X/2,X]|}\ge 2c_f(\log\log X)^{-\frac{1}{80}}.\end{align}

Let $I_1=(X/2,X]$. Let $b_1=0$ and $W_1=1$. Let
\begin{align}\label{valuedelta1}\delta_1=c_f(\log\log X)^{-\frac{1}{80}}.\end{align}
We deduce that
\begin{align*}\sum_{x\in \mathcal{A}\cap I_1}\Lambda( x)-\delta_1\sum_{x\in I_1}\Lambda(x)=\Big(\sum_{p\in \mathcal{A}} 1-\delta_1\sum_{p\in I_1}1\Big)\log X+O(\frac{X}{\log X}),\end{align*}
and by \eqref{density2},
\begin{align*}\sum_{x\in \mathcal{A}\cap I_1}\Lambda( x)-\delta_1\sum_{x\in I_1}\Lambda(x)\ge c_fX(\log\log X)^{-\frac{1}{80}}+O(\frac{X}{\log X}).\end{align*}
Therefore, we conclude
$$\delta_{b_1,W_1;I_1}(\mathcal{A})\ge \delta_1.$$
On applying Lemma \ref{lemmadensityincrease} iteratively, we can find $b_m,W_m,I_m,\delta_m$ ($m=1,2,\cdots $) satisfying
\begin{align}\label{iterative}W_{m+1}\le \alpha_0^{-1}\delta_m^{-21}W_m, \ \ |I_{m+1}|\ge \alpha_0\delta_m^{70}|I_m|,\ \ \ \delta_{m+1}\ge \delta_m(1+\alpha_0\delta_m^{70})\end{align}
and
$$(b_m,W_m)=1, \ \ \delta_{b_m,W_m;I_m}(\mathcal{A})\ge \delta_m$$
provided that in this process
\begin{align}\label{conditioniterative}W_m\le \log X\ \ \textrm{ and } \ \ |I_m|\ge \frac{X}{\log X}.\end{align}
The number $\alpha_0$ is the one in Lemma \ref{lemmadensityincrease}. We may assume $0<\alpha_0\le 1$ with $\log (\alpha_0^{-1})\le \alpha_0^{-0.1}$.

We deduce from \eqref{iterative} that
$$W_{m+1}\le \alpha_0^{-1}\delta_1^{-21}W_m, \ \ |I_{m+1}|\ge \alpha_0\delta_1^{70}|I_m|,$$
and
$$W_{m+1}\le \alpha_0^{-m}\delta_1^{-21m}W_1, \ \ |I_{m+1}|\ge \alpha_0^{m}\delta_1^{70m}|I_1|.$$
In view of the value of $\delta_1$ in \eqref{valuedelta1}, one has
\begin{align}\label{inequalitylast}\Big(\alpha_0^{-1}\delta_1^{-70}\Big)^{4\alpha_0^{-1}\delta_1^{-71}}\le (\log X)^{1/2}.\end{align}
Therefore, the inequalities in \eqref{conditioniterative} hold if $m\le \frac{4}{\alpha_0}\delta_1^{-71}$.

By \eqref{iterative},
$$\delta_{m+1}\ge \delta_{1}(1+\alpha_0\delta_1^{70})^{m}.$$
Then we observe $\delta_{m+1}>1$ if $m>\frac{2}{\alpha_0}\delta_1^{-71}$. This is a contradiction and the proof of Theorem \ref{theorem1} is complete.

\vskip4mm
\end{document}